\documentclass[10pt]{amsart}

\usepackage[plainpages = false, pdfpagelabels, 
						 bookmarks,
						 bookmarksopen = true,
						 bookmarksnumbered = true,
						 breaklinks = true,
						 linktocpage,
						 colorlinks = true,
						 linkcolor = blue,
						 urlcolor  = blue,
						 citecolor = red,
						 anchorcolor = green,
						 hyperindex = true,
						 hyperfigures
						 ]{hyperref} 
\usepackage{graphicx}
\usepackage[utf8]{inputenc}
\usepackage[english]{babel}
\usepackage[T1]{fontenc}

\usepackage{type1cm}

\usepackage{amsmath, amsfonts, amsthm, amscd, amssymb}
\usepackage[shortlabels]{enumitem}
\usepackage[normalem]{ulem}
\usepackage{extarrows}
\usepackage[all]{xy}

\usepackage[a4paper]{geometry}

\usepackage{abbreviazioni}
\usepackage{stile_en}

\numberwithin{equation}{section}

\title{On the density of zeros of linear combinations of Euler products for $\sigma>1$}
\author{Mattia Righetti}
\date{}
\address{Dipartimento di Matematica, Universit\`{a} di Genova, via Dodecaneso 35, 16146, Genoa, Italy}
\email{righetti@dima.unige.it}           

\begin{document}

\begin{abstract}
It has been conjectured that the real parts of the zeros of a linear combination of two or more $L$-functions are dense in the interval $[1,\sigma^*]$, where $\sigma^*$ is the least upper bound of the real parts of such zeros. In this paper we show that this is not true in general. Moreover, we describe the optimal configuration of the zeros of linear combinations of orthogonal Euler products by showing that the real parts of such zeros are dense in subintervals of $[1,\sigma^*]$ whenever $\sigma^*>1$.  
\end{abstract}

\maketitle

\section{Introduction}

Let $L(s)$ be a Dirichlet series and let $\sigma^*=\sigma^*(L)$ be the least upper bound of the real parts of the zeros of $L(s)$. Then it is well known that $\sigma^*$ is finite (see e.g. Titchmarsh \cite[\S9.41]{titchmarsh2}). For the Riemann zeta function $\zeta(s)$ we know that $\sigma^*\leq 1$, and it is expected that the Riemann Hypothesis holds, i.e. $\sigma^*= \frac{1}{2}$. A similar situation is expected for many Euler products (see e.g. Selberg \cite{selberg}).

On the other hand, we have recently proved \cite{righetti}, for a large class of $L$-functions with a polynomial Euler product, that non-trivial linear combinations $L(s)$ have zeros for $\sigma>1$. This is not surprising since many examples of such linear combinations were already known to have zeros for $\sigma>1$ from work of Davenport and Heilbronn \cite{davenport1,davenport2} on the Hurwitz and Epstein zeta functions. We also refer to later works of Cassels \cite{cassels}, Conrey and Ghosh \cite{conreyghosh}, Saias and Weingartner \cite{saias}, and Booker and Thorne \cite{booker}.

In certain cases something more is known, namely Bombieri and Mueller \cite{bombierimueller} have shown that the real parts of the zeros of certain Epstein zeta functions are dense in the interval $[1,\sigma^*]$. Note that these functions may be written as a linear combination of two Hecke $L$-functions. Other examples of linear combinations with this property may be found in Bombieri and Ghosh \cite{bombierighosh}, although not explicitly stated.

One might expect this to hold in general for the real parts of the zeros of linear combinations of two or more $L$-functions (see Bombieri and Ghosh \cite[p. 230]{bombierighosh}). However this is too much to hope for as one can see from the following.

\begin{theorem}\label{theorem:one_hole} Let $N\geq 2$ be an integer and, for $j=1,\ldots,N$, let $F_j(s)=\sum_{n=1}^\infty a_j(n)n^{-s}$ be a not-identically-zero Dirichlet series absolutely convergent for $\sigma>1$. If $\sum_{j=1}^N |a_j(1)|\neq 0$, then there exist infinitely many $\bb{c}=(c_1,\ldots,c_N)\in\C^N$ such that the Dirichlet series $L_{\bb{c}}(s)=\sum_{j=1}^N c_jF_j(s)$ has no zeros in some vertical strip $\sigma_1<\sigma<\sigma_2$ with $1<\sigma_1<\sigma_2<\sigma^*(L_{\bb{c}})$.
\end{theorem}

\begin{remark*} This result is very general, but in particular may be applied to linear combinations of $L$-functions. Moreover, it is easy to show that the same argument works in general for $c$-values.
\end{remark*}

We can actually prove more, i.e. it is in general possible to construct Dirichlet series, given by a linear combination of $L$-functions, which have many distinct vertical strips without zeros.

\begin{theorem}\label{theorem:holes}
Let $k\geq 1$ be an integer and, for $j=1,\ldots,k+1$, let $F_j(s)=\sum_{n=1}^\infty a_j(n)n^{-s}$ be a Dirichlet series absolutely convergent for $\sigma>1$ with $a_j(1)\neq 0$. Suppose that
\begin{equation}\label{eq:det}
\det\begin{pmatrix}
a_1(1) & a_1(2) & \cdots & a_1(k+1)\\
a_2(1) & a_2(2) & \cdots & a_2(k+1)\\
\vdots & \vdots & \ddots & \vdots\\
a_{k+1}(1) & a_{k+1}(2) & \cdots & a_{k+1}(k+1)\\
\end{pmatrix}\neq 0.
\end{equation}
Then there exist infinitely many $\bb{c}\in \C^{k+1}$ such that the Dirichlet series $L_{\bb{c}}(s)=\sum_{j=1}^{k+1} c_jF_j(s)$ has at least $k$ distinct vertical strips without zeros in the region $1<\sigma<\sigma^*(L_{\bb{c}})$.
\end{theorem}
\begin{remark*} Note that the vectors $\bb{c}$ of Theorem \ref{theorem:one_hole} cover a cone (without the vertex $\bb0$) generated by the intersection of an open ball with an hyperplane in $\C^{k+1}$, while the vectors $\bb{c}$ of Theorem \ref{theorem:holes} contain a line passing through the origin $\bb0$, excluding the origin itself. See the proofs in Section \ref{section:holes} for details.
\end{remark*}

The proof of Theorem \ref{theorem:holes} is actually constructive and may be used to explicitly obtain coefficients. As a concrete example we apply it to $\zeta(s)$, $L(s,\chi_1)$ and $L(s,\conj{\chi_1})$, where $\chi_1$ is the unique Dirichlet character mod $5$ such that $\chi_1(2)=i$, which satisfy the hypotheses of Theorem \ref{theorem:holes}. Hence we obtain the Dirichlet series
$$L(s)=c_1L(s,\chi_1)+c_2 L(s,\conj{\chi_1})+c_3\zeta(s),$$
where
\begin{spliteq*}
c_1 &= -\frac{1}{L(8,\conj{\chi_1})}\frac{L(16,\conj{\chi}_1)\zeta(8)-L(8,\conj{\chi}_1)\zeta(16)}{L(16,\chi_1)\zeta(8)-L(8,\chi_1)\zeta(16)} =  -0.08260584\ldots - i0.99658995\ldots,\\
c_2 &= \frac{1}{L(8,\conj{\chi_1})} =   1.00000059\ldots + i0.00375400\ldots,\\
c_3 &= \frac{1}{L(8,\conj{\chi_1})}\frac{{L(8,\chi_1)}{L(16,\conj{\chi}_1)}-{L(8,\conj{\chi}_1)}{L(16,\chi_1)}}{{\zeta(8)}{L(16,\chi_1)}-{L(8,\chi_1)}{\zeta(16)}} = -0.91739597\ldots + i0.99283727\ldots\,.
\end{spliteq*}

In Figure \ref{fig:due_buchi} we see part of two distinct vertical strips without zeros of $L(s)$ within the vertical strip $1<\sigma<\sigma^*$. We recall that, by Saias and Weingartner \cite{saias}, there are zeros in the vertical strip $1<\sigma<1+\eta$ for some $\eta>0$. Actually, in \cite{righetti} we have remarked that, as a consequence of the technique used to prove the main result, the real parts of the zeros of non-trivial combinations of orthogonal $L$-functions are dense in a small interval $[1,1+\eta]$, for some $\eta>0$ (cf. \cite[Corollary 1]{righetti}).

\begin{figure}[htb]
	\centering
		\includegraphics[width=14cm]{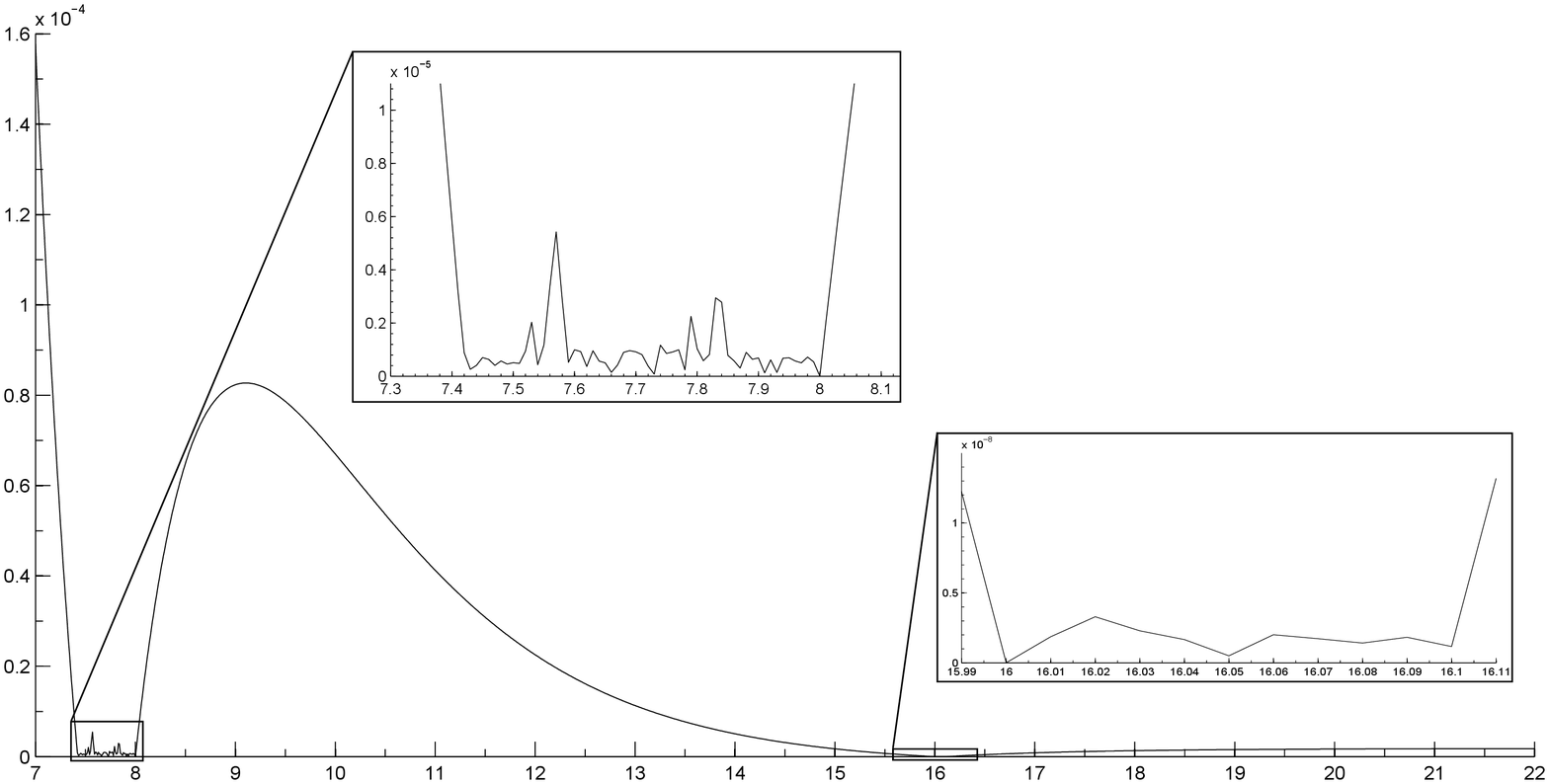}
	\caption{Approximate plot of $\min\limits_{t} |c_1L(\sigma+it,\chi_1)+c_2 L(\sigma+it,\conj{\chi_1})+ c_3\zeta(\sigma+it)|$ for $\sigma\in[7,22]$ and $t\in[0,2000]$ with step 0.01.}
	\label{fig:due_buchi}
\end{figure}

What we know in general about the real parts of the zeros of a Dirichlet series is the following result, which is a consequence of work of Jessen and Tornehave \cite{jessentornehave}.
\begin{theorem}\label{theorem:real_parts}
Let $L(s)=\sum_{n=n_0}^\infty \frac{a(n)}{n^s}$, with $a(n_0)\neq 0$, be uniformly convergent for $\sigma>\alpha$. Then in any vertical strip $\alpha<\alpha_1\leq\sigma\leq\sigma^*(L)$, $L(s)$ has only a finite number of zero-free vertical strips and a finite number of isolated vertical lines containing zeros.\\
In particular, if $\rho_0=\beta_0+i\gamma_0$ is a zero of $L(s)$ with $\beta_0>\alpha$, then either $\sigma=\beta_0$ is an isolated vertical line as above or there exist $\sigma_1\leq \beta_0\leq \sigma_2$, with $\sigma_1<\sigma_2$, such that the set
$$\{\beta\in[\sigma_1,\sigma_2]\mid \exists \gamma\in\R\hbox{ such that }L(\beta+i\gamma)=0\}$$
is dense in $[\sigma_1,\sigma_2]$.
\end{theorem}
The first statement is a reinterpretation of Theorem 31 of Jessen and Tornehave \cite{jessentornehave} in view of Theorem 8 of \cite{jessentornehave}. The second statement follows from the first one by a simple set-theoretic argument.

In this paper we also prove that a linear combination of \emph{orthogonal} (see Definition \ref{definition:orthogonality}) Euler products has no isolated vertical lines containing zeros. As in \cite{righetti} we work in an axiomatic setting. Given a complex function $F(s)$ we consider the following properties:
\begin{enumerate}[(I)]
\item\label{hp:DS} $\displaystyle F(s)=\sum_{n=1}^\infty \frac{a_F(n)}{n^s}$ is absolutely convergent for $\sigma>1$;
\item\label{hp:EP} $\displaystyle \log F(s) = \sum_p \sum_{k=1}^\infty\frac{b_F(p^k)}{p^{ks}}$ is absolutely convergent for $\sigma>1$, with $|b_F(p^k)|\ll p^{k\theta}$ for every prime $p$ and every $k\geq 1$, for some $\theta<\frac{1}{2}$;
\item\label{hp:RC} for any $\eps>0$, $|a_F(n)|\ll n^{\eps}$ for every $n\geq 1$.
\end{enumerate}

\begin{remark*}
If $F(s)$ satisfies \ref{hp:EP}, then $F(s)\neq 0$ for $\sigma>1$, i.e. $\sigma^*(F)\leq 1$.
\end{remark*}

\begin{definition}\label{definition:orthogonality}
For any integer $N\geq 1$, we say that $F_1(s),\ldots,F_N(s)$ satisfying \ref{hp:DS} are \emph{orthogonal} if
\begin{equation}\label{eq:SC}
\sum_{p\leq x} \frac{a_{F_i}(p)\conj{a_{F_j}(p)}}{p} = (m_{i,j}+o(1))\log\log x,\quad x\rightarrow \infty,
\end{equation}
with $m_{i,i}>0$ and $m_{i,j}=0$ if $i\neq j$.
\end{definition}

We can now state the main theorems. We consider separately the cases $N=2$ and $N\geq3$ since they are handled in different ways and yield different results, although the underling idea is the same.

\begin{theorem}\label{theorem:main1}
Let $F_1(s),\,F_2(s)$ be orthogonal functions satisfying \ref{hp:DS} and \ref{hp:EP}, $c_1,c_2\in\C\setminus\{0\}$, and $L(s)=c_1F_1(s)+c_2F_2(s)$. Then $L(s)$ has no isolated vertical lines containing zeros in the half-plane $\sigma>1$.
\end{theorem}

\begin{theorem}\label{theorem:main2}
Let $N\geq 3$ be an integer,  $c_1,\ldots,c_N\in\C\setminus\{0\}$, $c\in\C$, and $F_1(s),\ldots,F_N(s)$ orthogonal functions satisfying \ref{hp:DS}, \ref{hp:EP} and \ref{hp:RC}. If we write $L(s)=\sum_{j=1}^N c_jF_j(s)-c$, then $L(s)$ has no isolated vertical lines containing zeros in the half-plane $\sigma>1$.
\end{theorem}

\begin{remark*}
Note that orthogonality is necessary in Theorems \ref{theorem:main1} and \ref{theorem:main2} as it is shown by the following simple example
$$(1-2^{-s})\zeta(s)-\frac34\zeta(s)=\left(\frac{1}{4}-\frac{1}{2^s}\right)\zeta(s),$$
which clearly vanishes, in the half-plane of absolute convergence $\sigma>1$, only on the vertical line $\sigma=2$.
\end{remark*}

\begin{remark*}
There are some differences between the axioms that in \cite{righetti} define the class $\mathcal{E}$ and the above axioms \ref{hp:DS}, \ref{hp:EP} and \ref{hp:RC}, so that in principle we cannot say that the results that we obtained in \cite{righetti} may be applied here or \emph{viceversa}. In particular, for a Dirichlet series $L(s)$ as in Theorems \ref{theorem:main1} and \ref{theorem:main2} we don't know if $\sigma^*(L)>1$. However most of the known families of $L$-functions satisfy, or are supposed to satisfy, both the axioms of $\mathcal{E}$ and \ref{hp:DS}, \ref{hp:EP} and \ref{hp:RC}.
\end{remark*}

From Theorems \ref{theorem:main1} and \ref{theorem:main2} we obtain the following simple consequence.

\begin{corollary}\label{corollary:main_coro}
Let $L(s)$ be as in Theorems \ref{theorem:main1} or \ref{theorem:main2}. If $\sigma^*(L)>1$, then there exists $\eta>0$ such that the set
$$\{\beta\in[\sigma^*(L)-\eta,\sigma^*(L)]\mid \exists \gamma\hbox{ such that }L(\beta+i\gamma)=0\}$$
is dense in $[\sigma^*(L)-\eta,\sigma^*(L)]$.
\end{corollary}
\begin{proof}
If $\sigma^*=\sigma^*(L)$ is itself the real part of a zero, the result follows immediately from the second part of Theorem \ref{theorem:real_parts} and Theorems \ref{theorem:main1} and \ref{theorem:main2}, choosing $\eta=\sigma^*-\sigma_1>0$ (while $\sigma_2=\sigma^*$ by definition). Suppose otherwise that $\sigma^*$ is not the real part of a zero. Then by definition, for every $\eps>0$ there exist $\beta_\eps\in(\sigma^*-\eps,\sigma^*)$ and $\gamma_\eps\in\R$ such that $L(\sigma_\eps+i\gamma_\eps)=0$. Hence $\sigma^*$ is the limit point of the real part of certain zeros of $L(s)$. Note that in general if $L(\sigma+it)\neq 0$, then either for any $\eps>0$ there exist $\beta_\eps$ with $|\sigma-\beta_\eps|<\eps$ and $\gamma_\eps\in\R$ such that $L(\beta_\eps+i\gamma_\eps)=0$, i.e. $\sigma$ is the limit point of the real part of certain zeros of $L(s)$, or there exists an open interval $(\sigma-\delta,\sigma+\delta)$, for some $\delta>0$, which does not contain any real part of the zeros. Since by Theorem \ref{theorem:real_parts} the number of such intervals is finite for $\sigma\in(\sigma^*-\eps,\sigma^*)$ for every $\eps>0$, we can take $\eta=\eps$ small enough so that there are none.
\end{proof}

By Theorems \ref{theorem:one_hole} and \ref{theorem:holes} we see that Theorems \ref{theorem:main1} and \ref{theorem:main2} are optimal, in the sense that without conditions on the coefficients $\bb{c}$ we cannot expect stronger results on the density of the real parts of the zeros. As an example we see in Figure \ref{fig:DH_taumeno} that the real parts of the zeros of
$$f(s,\tau)=\frac{1}{2}\left[(1-i\tau)L(s,\chi_1) + (1+i\tau)L(s,\conj{\chi_1})\right], \qquad \tau=-\frac{1+\sqrt{5}}{2}-\sqrt{1+\left(\frac{1+\sqrt{5}}{2}\right)^2},$$
are dense up to $\sigma^* = 2.3822861089\ldots$, as we have proved in our Ph.D. thesis \cite[p. 66]{righettiphd}. On the other hand, we see that the real parts of the zeros of $L(s,\chi_1)-cL(s,\conj{\chi_1})$, where
$$c = \frac{L(8,\chi_1)}{L(8,\conj{\chi_1})}= 0.99997181\ldots + i0.00750790\ldots,$$
are dense close to $\sigma=1$, as we know from work of Saias and Weingartner \cite{saias} (cf. \cite[Corollary 1]{righetti}), there are no zeros with real part in the interval $[2,7]$, but $s=8$ is a zero.
\begin{figure}[thb]
	\centering
		\includegraphics[width=14cm]{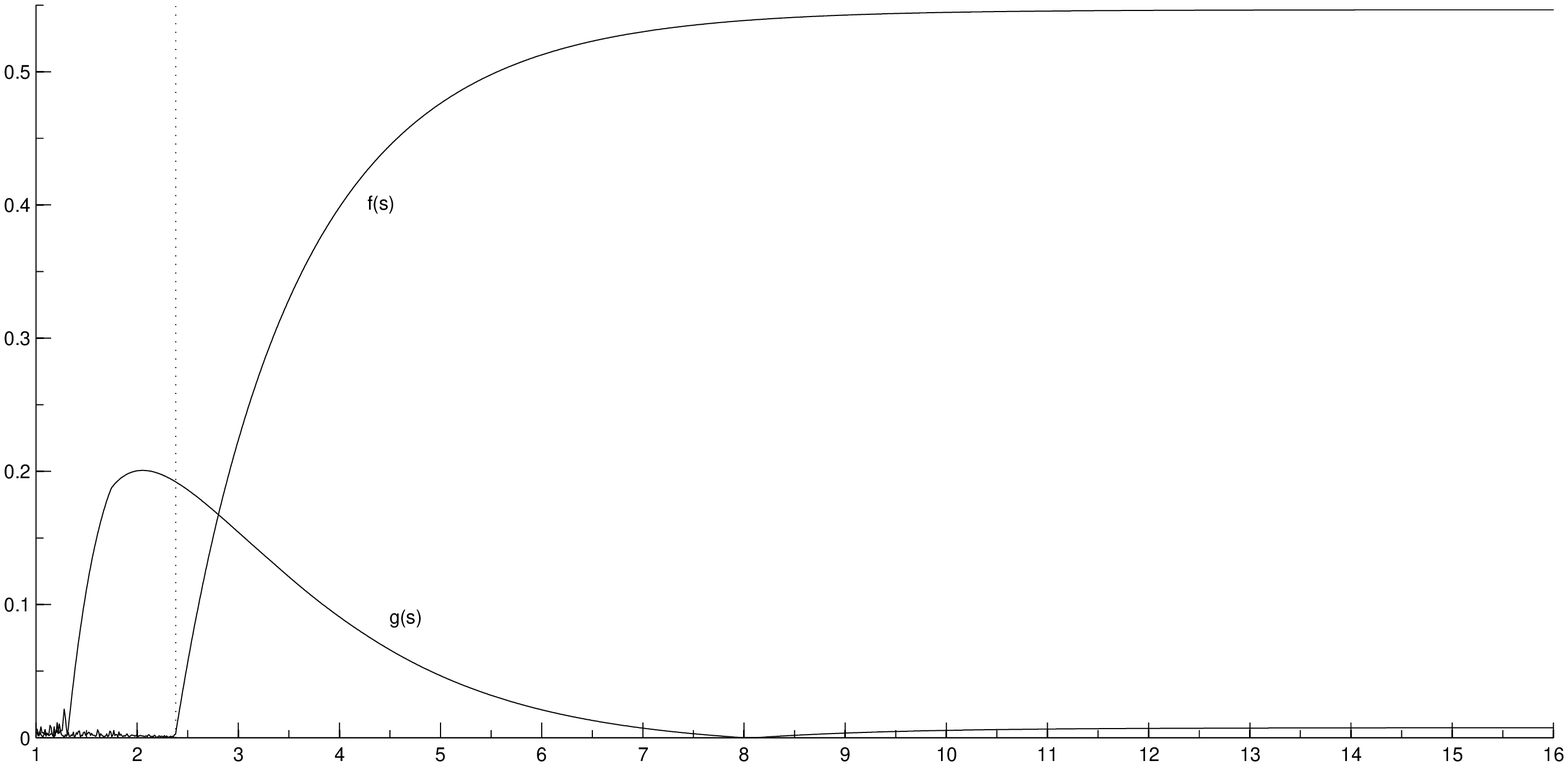}
	\caption{Approximate plot of $f(\sigma)=\min\limits_{t}\abs{\frac{L(\sigma+it,\chi_1)}{L(\sigma+it,\conj{\chi_1})}+\frac{1+i\tau}{1-i\tau}}$ and of $g(\sigma)=\min\limits_{t}\abs{\frac{L(\sigma+it,\chi_1)}{L(\sigma+it,\conj{\chi_1})}-\frac{L(8,\chi_1)}{L(8,\conj{\chi_1})}}$, where $\sigma\in[1.01, 16.01]$ and $t\in[0,2000]$ with step 0.01.}
	\label{fig:DH_taumeno}
\end{figure}

Note that the above function $f(s,\tau)$ is of the Davenport--Heilbronn type studied by Bombieri and Ghosh \cite{bombierighosh}. As we already remarked, Bombieri and Ghosh do not say whether these Davenport--Heilbronn type functions do have the property that the real parts of their zeros are dense in $[1,\sigma^*]$. However, in \cite[Theorem 4.1.3]{righettiphd} we gave necessary and sufficient conditions on the coefficients of these Davenport--Heilbronn type functions for this to happen, namely
\begin{theorem}
Suppose that $\xi\in\R$, $q$ is a positive integer and $\chi_0$ is the principal character mod $q$. Then there exists $\xi_{\rm max}(q)$, such that the real parts of the zeros for $\sigma>1$ of
$$f(s,\xi,q)=\frac12[(1-i\xi)L(s,\chi_1\chi_0)+(1+i\xi)L(s,\conj{\chi_1}\chi_0)]$$
are dense in the interval $[1,\sigma^*(\xi,q)]$ if and only if $|\xi|\leq \xi_{\rm max}(q)$. In particular, if $6\nondiv q$, then it is sufficient to take $|\xi|\leq 6.5851599$.
\end{theorem}
\begin{proof} The proof is a continuation of the proof of Bombieri and Ghosh \cite[Theorem 7]{bombierighosh} using results of Kershner \cite[Theorems II--III]{kershner1} on the support function of the inner border of the sum of convex curves.
\end{proof}

Theorems \ref{theorem:main1} and \ref{theorem:main2} are obtained by suitably adapting the works of Bohr and Jessen \cite{bohrjessen1,bohrjessen2}, Wintner and Jessen \cite{jessenwintner}, Jessen and Tornehave \cite{jessentornehave}, Borchsenius and Jessen \cite{borchseniusjessen}, and Lee \cite{lee} on the value distribution of Dirichlet series. The proofs will be given in Sections \ref{section:proof_main1} and \ref{section:proof_main2} respectively.

Note that in Theorems \ref{theorem:main1} and \ref{theorem:main2} we don't ask for a functional equation or meromorphic continuation to the whole complex plane. However, in the concrete case these are always known to hold, so one might ask what happens if one adds these conditions. On account of this we show that Theorem \ref{theorem:one_hole} may be generalized so that the resulting Dirichlet series is an $L$-function with functional equation and, of course, without Euler product. We hence consider functions $F(s)$ satisfying \ref{hp:DS} and
\begin{enumerate}[(I)]\setcounter{enumi}{3}
\item\label{hp:AC} $(s-1)^m F(s)$ is an analytic continuation as an entire function of finite order for some $m\geq 0$;
\item\label{hp:FE} $F(s)$ satisfies a functional equations of the form $\Phi(s)=\omega \conj{\Phi(1-\conj{s})}$, where $|\omega|=1$ and
\begin{equation*} 
\Phi(s) = Q^s \prod_{j=1}^r\Gamma\left(\lambda_j s+\mu_j\right)F(s) = \gamma(s)F(s),
\end{equation*}
say, with $r\geq 0$, $Q>0$, $\lambda_j>0$ and $\re\mu_j\geq 0$; 
\end{enumerate}
although such requirements can actually be relaxed. 
\begin{theorem}\label{theorem:func_eq_hole}
Let $N\geq 3$ be an integer, $(r,Q,\bb{\lambda},\bb{\mu})$ fixed parameters, and $F_1(s),\ldots,F_N(s)$ functions satisfying \ref{hp:DS}, \ref{hp:EP}, \ref{hp:AC} and \ref{hp:FE} for some $|\omega_j|= 1$, $j=1,\ldots,N$. Suppose furthermore that $\omega_h\neq \omega_k$ for some $h,k\in\{1,\ldots,N\}$. Then there exist infinitely many $\bb{c}\in\C^N$ such that $L_{\bb{c}}(s)=\sum_{j=1}^N c_jF_j(s)$ satisfies \ref{hp:AC}, \ref{hp:FE} and has no zeros in some vertical strip $\sigma_1<\sigma<\sigma_2$ with $1<\sigma_1<\sigma_2<\sigma^*(L_{\bb{c}})$.
\end{theorem}

To give a concrete example of the above result, we fix an integer $q\geq 7$, square-free, $(q,6)=1$ and $q\not\equiv 2\smod{4}$, and consider the Dirichlet $L$-functions associated to distinct primitive characters $\chi\smod{q}$. Then, we know that there are $\varphi^*(q) = \prod_{p|q} (p-2)$ distinct primitive characters $\chi\smod{q}$ and at least half of them have the same \emph{parity}, i.e. the sign of $\chi(-1)$. We denote with $\mathcal{W}(q)$ the set of such characters and we have that $|\mathcal{W}(q)|\geq 3$. As a consequence of Theorem 1 of Kaczorowski, Molteni and Perelli \cite{kmp4}, we have that $\omega_{\chi_1}\neq \omega_{\chi_2}$ if $\chi_1\neq \chi_2$ for $\chi_1,\chi_2\in\mathcal{W}(q)$, so we may apply Theorem \ref{theorem:func_eq_hole} to the Dirichlet $L$-functions associated to distinct characters of $\mathcal{W}(q)$.

On the other hand, we mention that Bombieri and Hejhal \cite{bombierihejhal} have shown that, under the Generalized Riemann Hypothesis and a weak pair correlation of the zeros, linear combinations with real coefficients of Euler products with the same functional equation have asymptotically almost all of their zeros on the line $\sigma=\frac{1}{2}$. For $\frac{1}{2}<\sigma<1$, it is known that \emph{joint universality} of a $N$-uple of $L$-functions implies that the real parts of the zeros of any linear combination of these $L$-functions are dense in $(\frac{1}{2},1)$ (see e.g. Bombieri and Ghosh \cite[p. 230]{bombierighosh}). Joint universality is known to hold for many families of $L$-functions and recently Lee, Nakamura and Pa\'nkowski \cite{lnp} have shown that this property holds in an axiomatic setting such as the Selberg class under a strong Selberg orthonormality conjecture.

To give concrete examples of families of $L$-functions satisfying the properties required by Theorems \ref{theorem:main1} and \ref{theorem:main2} we refer to \cite{righetti} for Artin $L$-functions, automorphic $L$-functions and the Selberg class. Here we only recall that the relevant analytic properties of the automorphic $L$-functions and their orthogonality can be found in the papers of Rudnick and Sarnak \cite{rudnick}, Iwaniec and Sarnak \cite{iwaniec}, Bombieri and Hejhal \cite{bombierihejhal}, Kaczorowski and Perelli \cite{kaczper2}, Kaczorowski, Molteni and Perelli \cite{kmp2}, Liu and Ye \cite{liu2}, and Avdispahi\'{c} and Smajlovi\'{c} \cite{avdispahic}. Moreover, we refer to Selberg \cite{selberg} and the surveys of Kaczorowski \cite{kaczorowski}, Kaczorowski and Perelli \cite{kaczper1}, and Perelli \cite{perelli1} for a thorough discussion on the Selberg class.

\vspace{5pt}For the computations we have used the software packages PARI/GP \cite{pari} and MATLAB$^\circledR$. These were made by truncating the Dirichlet series to the first $70\,000$ terms, which guarantees accuracy to eight decimal places for the values given above.

\begin{acknowledgements}
This paper is part of my Ph.D. thesis at the Department of Mathematics of the University of Genova. I express my sincere gratitude to my supervisor Professor Alberto Perelli for his support and for the many valuable suggestions he gave me. I thank Professor Giuseppe Molteni for carefully reading the manuscript and suggesting several improvements. I also wish to thank Professor Enrico Bombieri for an enlightening discussion on this topic and for suggestions concerning the paper.
\end{acknowledgements}
\section{Radii of convexity of power series}

Let $F(s)$ be a function satisfying \ref{hp:DS} and \ref{hp:EP}. Then we can write $F(s)$ as a absolutely convergent Euler product $F(s)=\prod_p F_p(s)$ for $\sigma>1$, where the local factor $F_p(s)$ is determined by $\log F_p(s)=\sum_{k=1}^\infty\frac{b_F(p^k)}{p^{ks}}$. Then, in most of the results on the value distribution of $F(s)$ for some fixed $\sigma$, a fundamental ingredient is the convexity of the curves $\log F_p(\sigma+it)$, $t\in\R$, at least for infinitely many primes $p$. In this section we collect and prove some results on this matter which will be needed later.

Let $\mathcal{A}$ be the class of functions $f(z)=z+\sum_{n=2}^\infty b(n)z^n$ which are regular on $D=\{|z|<1\}$. Let $\mathcal{F}$ be any subclass of $\mathcal{A}$, then we write $r_c(\mathcal{F})$ for the largest $r$, with $0<r\leq 1$, such that $f(\{|z|<r\})$ is convex. In \cite{yamashita}, Yamashita proved the following result.

\begin{proposition}[Yamashita {\cite[Theorem 2]{yamashita}}]\label{proposition:yamashita} Let $\mathcal{B}=\{f\in\mathcal{A}\mid |b(n)|\leq n,\: n\geq 2\}$. Then $r_c(B)\geq R_1$, where $R_1$ is the smallest real root in $(0,1)$ of $2(1-X)^4=1+4X+X^2$.\\
Let $K>0$ and $\mathcal{G}(K)=\{f\in\mathcal{A}\mid |b(n)|\leq K,\: n\geq 2\}$. Then $r_c(\mathcal{G}(K))\geq R_2(K)$, where $R_2(K)$ is the smallest real root in $(0,1)$ of $X^3-3X^2+4X=(1-X)^3/K$.
\end{proposition}

The proof of the above proposition is actually a simple consequence of the following result of Alexander and Remak (see Goodman \cite[Theorem 1]{goodman}).

\begin{theorem}[Alexander--Remak]\label{theorem:AR}
If $f(z)=z+\sum_{n=2}^\infty b(n)z^n\in\mathcal{A}$ and
$$\sum_{n=2}^\infty n^2|b(n)|\leq 1,$$
then $f(D)$ is convex.
\end{theorem}

Adapting Yamashita's proof (cf. \cite[\S2]{yamashita}) we obtain the following.

\begin{proposition}\label{proposition:radius_convexity} Let $K>0$ and $\mathcal{H}(K) = \{f\in\mathcal{A}\mid |b(n)|\leq Kn^2,\: n\geq 2\}$. Then $r_c(\mathcal{H}(K))\geq R_3(K)$, where $R_3(K)$ is the smallest root in $(0,1)$ of $X^5-5X^4+11X^3+X^2+16X=(1-X)^5/K$.
\end{proposition}

\begin{remark}\label{remark:radii}
Note that $R_3(K)$ is a strictly decreasing function of $K$, with
$$\sup_{K>0}R_3(K) = \lim_{K\rightarrow 0^+}R_3(K)=1\qquad\hbox{ and }\qquad\inf_{K>0}R_3(K) = \lim_{K\rightarrow +\infty}R_3(K)=0.$$
Moreover, for any $K>0$ we have $R_3(K)\leq R_2(K)$.
\end{remark}

\begin{proof}[Proof of Proposition \ref{proposition:radius_convexity}]
For $f(z)=z+\sum_{n=2}^\infty b(n)z^n\in\mathcal{H}(K)$ and any $r\leq R_3=R_3(K)$ we have 
$$\sum_{n=2}^\infty n^2|b(n)|r^{n-1}\leq K\sum_{n=2}^\infty n^4R_3^{n-1}=K\frac{R_3^5-5R_3^4+11R_3^3+R_3^2+16R_3}{(1-R_3)^5}=1,$$
where the last equality follows from the fact that is chosen $R_3$ as the smallest real root in $(0,1)$ of $X^5-5X^4+11X^3+X^2+16X=(1-X)^5/K$. Therefore we can apply Theorem \ref{theorem:AR} to $h(z)=r^{-1}f(rz)$, which is thus convex on $|z|<1$. Hence $f(\{|z|<r\})$ is convex for any $r\leq R_3$ and thus $R_3\leq r_c(\mathcal{H}(K))$.
\end{proof}

From this we obtain an explicit version of Theorem 13 of Jessen and Wintner \cite{jessenwintner} and Lemma 2.5 of Lee \cite{lee}.

\begin{proposition}\label{proposition:bound_trasf_fourier}
Let $N$ be a fixed positive integer, $G_j(z)=\sum_{n=1}^\infty a_j(n)z^n$, $j=1,\ldots,N$, and suppose there exist positive real numbers $\rho_j$ and $K_j$ such that $|a(n)|\leq K_j\rho_j^{1-n}$ for every $n\geq 2$. For any $\boldsymbol{y}=(y_1,\ldots,y_J)\in\C^N$, define 
$$g(r,\theta,\boldsymbol{y}) = \sum_{j=1}^N \re\left(G_j(r\exp{2\pi i\theta})\conj{y_j}\right),$$
where $0<r<\min_j \rho_j$ and $\theta\in[0,1]$. If $\sum_{j=1}^N \conj{y_j}a_j(1)\neq 0$, then there exists a positive constant $C$ such that for any $\delta>0$ we have
\begin{equation}\label{eq:bound_trasf_fourier}
\abs{\int_0^1 \exp{i g(r,\theta,\boldsymbol{y})}d\theta}\leq \frac{24}{\sqrt{C\delta r\norma{\boldsymbol{y}}}}
\end{equation}
for every $0<r\leq R_3\!\left(\frac{1}{\delta}\sqrt{\sum_j|K_j|^2}\right)\cdot\min_j\rho_j$ and every $\boldsymbol{y}$ such that $\abs{\sum_{j=1}^N \conj{y_j}a_j(1)}\geq \delta\norma{\boldsymbol{y}}>0$.  
\end{proposition}
\begin{proof} The proof is a combination of Theorems 12 and 13 of Jessen and Wintner \cite{jessenwintner} and Lemma 2.5 of Lee \cite{lee}, and we use the aforementioned results to obtain explicit constants. Consider the power series
$$f(z) = \sum_{n=1}^\infty \left(\sum_{j=1}^N \conj{y_j}a_j(n)\right)z^n\quad\hbox{and}\quad h(z) = \sum_{n=1}^\infty n^2\left(\sum_{j=1}^N \conj{y_j}a_j(n)\right)z^n.$$
Since, by hypothesis and Cauchy-Schwarz inequality, we have 
\begin{equation}\label{eq:bound_coeff}
\abs{\sum_{j=1}^N \conj{y_j}a_j(n)}\leq \frac{\norma{\boldsymbol{y}}\sqrt{\sum_j|K_j|^2}}{(\min_j \rho_j)^{n-1}}\qquad \forall n\geq 2,
\end{equation}
$f(z)$ and $h(z)$ are both holomorphic for $|z|< \min_j \rho_j$ and, by definition, we have
$$g(r,\theta,\boldsymbol{y}) = \re f(r\exp{2\pi i \theta}) \quad\hbox{and}\quad g''(r,\theta,\boldsymbol{y}) = \frac{\partial^2}{\partial \theta^2}g(r,\theta,\boldsymbol{y})= -4\pi^2 \re h(r\exp{2\pi i \theta}).$$
By Proposition \ref{proposition:yamashita} we have that $f(r\exp{2\pi i \theta})$ is a parametric representation of a convex curve if $r\leq R_2\!\left(\frac{\norma{\boldsymbol{y}}\sqrt{\sum_j|K_j|^2}}{\abs{\sum_{j=1}^N \conj{y_j}a_j(1)}}\right)\cdot\min_j\rho_j$. Indeed, substituting $w=z/\min_j \rho_j$, we have
$$\widetilde{f}(w) = \frac{f(z/\min_j \rho_j)}{(\min_j \rho_j)\left(\sum_{j=1}^N \conj{y_j}a_j(1)\right)} =w + \sum_{n=2}^\infty (\min_j \rho_j)^{n-1}\left(\frac{\sum_{j=1}^N \conj{y_j}a_j(n)}{\sum_{j=1}^J \conj{y_j}a_j(1)}\right)w^n$$
and, by \eqref{eq:bound_coeff}, $\widetilde{f}(w)\in \mathcal{G}\!\left(\frac{\norma{\boldsymbol{y}}\sqrt{\sum_j|K_j|^2}}{\abs{\sum_{j=1}^J \conj{y_j}a_j(1)}}\right)$. Analogously, by Proposition \ref{proposition:radius_convexity} we have that $h(r\exp{2\pi i \theta})$ is a parametric representation of a convex curve if $r\leq R_3\!\left(\frac{\norma{\boldsymbol{y}}\sqrt{\sum_j|K_j|^2}}{\abs{\sum_{j=1}^N \conj{y_j}a_j(1)}}\right)\cdot\min_j\rho_j$. Therefore, by Remark \ref{remark:radii}, both $f(r\exp{2\pi i \theta})$ and $h(r\exp{2\pi i \theta})$ are parametric representations of convex curves for any fixed $r\leq R_3\!\left(\frac{\norma{\boldsymbol{y}}\sqrt{\sum_j|K_j|^2}}{\abs{\sum_{j=1}^N \conj{y_j}a_j(1)}}\right)\cdot\min_j\rho_j$. This implies that both $g(r,\theta,\boldsymbol{y})$ and $g''(r,\theta,\boldsymbol{y})$ have exactly two zeros mod $1$. By the mean value theorem, we have that also $g'(r,\theta,\boldsymbol{y})$ has exactly two zeros mod $1$, which separate those of $g''(r,\theta,\boldsymbol{y})$. Note that the zeros of $g'(r,\theta,\boldsymbol{y})$ and $g''(r,\theta,\boldsymbol{y})$ depend continuously on $r$ and $\boldsymbol{y}$ since $g'(r,\theta,\boldsymbol{y})$ and $g''(r,\theta,\boldsymbol{y})$ are continuous functions in each variable.

We now consider the mid-points of the four arcs mod $1$ determined by the zeros of $g'(r,\theta,\boldsymbol{y})$ and $g''(r,\theta,\boldsymbol{y})$. These mid-points clearly depend continuously on $r$ and $\boldsymbol{y}$, and divide $[0,1]$ into four arcs, namely $I_1$, $I_2$, $I_3$ and $I_4$, such that $I_1$ and $I_3$ each contain one zero of $g'(r,\theta,\boldsymbol{y})$, while $I_2$ and $I_4$ each contain one zero of $g''(r,\theta,\boldsymbol{y})$. By van der Corput's lemma for oscillatory integrals (see Titchmarsh \cite[Lemmas 4.2 and 4.4]{titchmarsh}) we have
$$\abs{\int_{I_2\cup I_4}\exp{ig(r,\theta,\boldsymbol{y})}d\theta}\leq \frac{8}{\displaystyle\min_{I_2\cup I_4} |g'(r,\theta,\boldsymbol{y})|}$$
and
$$\abs{\int_{I_1\cup I_3}\exp{ig(r,\theta,\boldsymbol{y})}d\theta}\leq \frac{16}{\displaystyle\sqrt{\min_{I_1\cup I_3} |g''(r,\theta,\boldsymbol{y})|}}.$$
Writing
$$g(r,\theta,\boldsymbol{y}) = r\abs{\sum_{j=1}^N \conj{y_j} a_j(1)} \cos\left(2\pi (\theta-\xi)\right)+r^2O(\norma{\boldsymbol{y}})$$
for some $\xi$, we see that by continuity there exists a positive constant $C$ such that
$$\frac{g'(r,\theta,\boldsymbol{y})}{r\abs{\sum_{j=1}^N \conj{y_j} a_j(1)}}\geq C\hbox{ on $I_2$ and $I_4$, and }\frac{g''(r,\theta,\boldsymbol{y})}{r\abs{\sum_{j=1}^N \conj{y_j} a_j(1)}}\geq C\hbox{ on $I_1$ and $I_3$}$$
for every $r\leq R_3\!\left(\frac{\norma{\boldsymbol{y}}\sqrt{\sum_j|K_j|^2}}{\abs{\sum_{j=1}^N \conj{y_j}a_j(1)}}\right)\cdot\min_j\rho_j$ and $\boldsymbol{y}\in\C^N$.\\
We fix $\delta>0$, $\boldsymbol{y}\neq \bb0$ such that $\abs{\sum_{j=1}^J \conj{y_j}a_j(1)}\geq \delta\norma{\boldsymbol{y}}$, $r\leq R_3\!\left(\frac{1}{\delta}\sqrt{\sum_j|K_j|^2}\right)\cdot\min_j\rho_j$, and we obtain
$$\abs{\int_0^1 \exp{i g(r,\theta,\boldsymbol{y})}d\theta}\leq \frac{8}{C\delta r\norma{\boldsymbol{y}}}+\frac{16}{\sqrt{C\delta r\norma{\boldsymbol{y}}}}.$$
Since $\frac{1}{C\delta r\norma{\boldsymbol{y}}}\leq\frac{1}{\sqrt{C\delta r\norma{\boldsymbol{y}}}}$ when $C\delta r\norma{\boldsymbol{y}}\geq 1$, then
$$\abs{\int_0^1 \exp{i g(r,\theta,\boldsymbol{y})}d\theta}\leq \frac{24}{\sqrt{C\delta r\norma{\boldsymbol{y}}}}\quad\hbox{for }\norma{\boldsymbol{y}}\geq \frac{1}{C\delta r}.$$
On the other hand, we clearly have that $\abs{\int_0^1 \exp{i g(r,\theta,\boldsymbol{y})}d\theta}\leq 1$, hence \eqref{eq:bound_trasf_fourier} holds whenever the RHS is $\geq 1$. Therefore the result follows from the simple fact that the RHS of \eqref{eq:bound_trasf_fourier} is $>24$ when $0<\norma{\boldsymbol{y}}< \frac{1}{C\delta r}$.
\end{proof}

\begin{theorem}\label{theorem:bound_trasf_fourier_primes}
Let $N$ be a positive integer and $F_1(s),\ldots, F_N(s)$  be orthogonal functions satisfying \ref{hp:DS} and \ref{hp:EP}. Then there exists a positive constant $A$ and infinitely many primes $p$ such that 
\begin{equation}\label{eq:bound_trasf_fourier_primes}
\abs{\int_0^1 \exp{i \sum_{j=1}^N \re\left(\conj{y_j}\log F_{j,p}(\sigma+i\frac{2\pi \theta}{\log p})\right)}d\theta}\leq \frac{A}{\sqrt{\norma{\boldsymbol{y}}}}p^{\frac{\sigma}{2}}
\end{equation}
for every $\sigma\geq 1$ and every $\boldsymbol{y}=(y_1,\ldots,y_N)\in\C^N\setminus\{\boldsymbol0\}$.
\end{theorem}
\begin{proof} We want to apply Proposition \ref{proposition:bound_trasf_fourier} to $G_j(z) = \displaystyle\sum_{n=1}^\infty \frac{b_{F_j}(p^n)}{\sqrt{m_{j,j}}}z^n$, $j=1,\ldots,N$. By \ref{hp:EP} there exist $K_{F_j}$ and $\theta_j<\frac{1}{2}$ such that for every prime $p$ and every $n\geq 2$ we have $|b_{F_j}(p^n)|\leq K_{F_j}p^{n\theta_{j}}\leq K_{F_j}p^{2(n-1)\theta_{j}}$, $j=1,\ldots,N$. Thus, for every prime $p$ we may take $K_j=K_{F_j}/\sqrt{m_{j,j}}$ and $\rho_j=p^{-2\theta_{j}}$, $j=1,\ldots,N$.\\
On the other hand, by orthogonality we have that for any $\bb{y}\neq \bb0$
$$\sum_{p\leq x}\frac{\abs{\frac{\conj{y_1}b_{F_1}(p)}{\sqrt{m_{1,1}}}+\cdots+\frac{\conj{y_N}b_{F_N}(p)}{\sqrt{m_{N,N}}}}^2}{p}\sim \norma{\boldsymbol{y}}^2\log\log x,\quad\hbox{as }x\rightarrow\infty.$$
In particular this implies that there are infinitely many primes $p$ such that
$$\abs{\frac{\conj{y_1}b_{F_1}(p)}{\sqrt{m_{1,1}}}+\cdots+\frac{\conj{y_N}b_{F_N}(p)}{\sqrt{m_{N,N}}}}\geq \norma{\boldsymbol{y}}$$
for every $\boldsymbol{y}\neq \boldsymbol{0}$. For each of such primes $p$ we take $r=p^{-\sigma}$ and $\delta=1$. Then Proposition \ref{proposition:bound_trasf_fourier} yields
\begin{equation}\label{eq:almost_there}
\abs{\int_0^1 \exp{i \sum_{j=1}^N\re\left(\frac{\conj{y_j}}{\sqrt{m_{j,j}}}\log F_{j,p}\left(\sigma + i\frac{2\pi \theta}{\log p}\right)\right)}d\theta}\leq \frac{24}{\sqrt{C \norma{\boldsymbol{y}}}}p^{\frac{\sigma}{2}}
\end{equation}
when
\begin{equation}\label{eq:radius}
p^{-\sigma}\leq R_3\!\left(2\sqrt{\sum_j\frac{|K_{F_j}|^2}{m_{j,j}}}\right)p^{-2\max_j \theta_{j}}
\end{equation}
and $\boldsymbol{y}\neq\bb0$. Note that \eqref{eq:radius} holds for every $\sigma\geq 1$ if $p$ is sufficiently large since $\max_j \theta_{j}<\frac{1}{2}$. Now, substituting $\boldsymbol{y'}=(y_1',\ldots,y_N')=\left(\frac{y_1}{\sqrt{m_{1,1}}},\ldots,\frac{y_N}{\sqrt{m_{N,N}}}\right)$ in \eqref{eq:almost_there} we obtain that there are infinitely many primes $p$ such that
$$\abs{\int_0^1 \exp{i \sum_{j=1}^N\re\left(\conj{y_j'}\log F_{j,p}\left(\sigma + i\frac{2\pi \theta}{\log p}\right)\right)}d\theta}\leq \frac{24}{\sqrt{C} \sqrt[4]{m_{1,1}|y_1'|^2+\cdots+m_{N,N}|y_N'|^2}}p^{\frac{\sigma}{2}}$$
for every $\sigma\geq1$ and every $\boldsymbol{y'}\in\C^N\setminus\{\boldsymbol0\}$. By the equivalence of norms in $\C^N$ we know that there exists a positive constants $D$ such that $\sqrt{m_{1,1}|y_1'|^2+\cdots+m_{N,N}|y_N'|^2}\geq D\norma{\boldsymbol{y'}}$, and hence the result follows immediately with $A=\frac{24}{\sqrt{DC}}$.
\end{proof}

\begin{remark}\label{remark:convexity_1_2}
From the proof we have that \eqref{eq:bound_trasf_fourier_primes} holds for $\sigma\geq1$ because $\max_j \theta_{j}<\frac{1}{2}$ by \ref{hp:EP}. Therefore if we had that $\max_j \theta_{j}<\frac{\kappa}{2}$ for some $0<\kappa<1$, we would immediately have that \eqref{eq:bound_trasf_fourier_primes} holds for every $\sigma\geq\kappa$.
\end{remark}

\section{On some distribution functions}\label{section:distribution_functions}

This section is an adaptation of Chapter II of Borchsenius and Jessen \cite{borchseniusjessen}. We will also use Theorem \ref{theorem:bound_trasf_fourier_primes} as well as Borchsenius and Jessen use Theorem 13 of Jessen and Wintner \cite{jessenwintner}. The particular distribution functions under investigation in this section may be found in Lee \cite{lee} and they will be used in Sections \ref{section:proof_main1} and \ref{section:proof_main2} for the proofs of Theorems \ref{theorem:main1} and \ref{theorem:main2}. The reader may wish to read Lee \cite{lee} for a brief introduction to the theory developed by Jessen and Tornehave \cite{jessentornehave} and Borchsenius and Jessen \cite{borchseniusjessen} and how it may be applied to linear combinations of Euler products.

Given a function $F(s)$ satisfying \ref{hp:DS} and \ref{hp:EP}, and a positive integer $n$ we write
$$F_n(s)=\prod_{m=1}^n F_{p_m}(s)\quad\hbox{and}\quad F_n(\sigma,\boldsymbol{\theta})=F_n(\sigma,\theta_1,\ldots,\theta_n)=\prod_{m=1}^n F_{p_m}\left(\sigma+i\frac{2\pi \theta_m}{\log p_m}\right),$$
where $p_m$ is the $m$-th prime and $F_p(s)$ is determined by $\log F_p(s)=\sum_{k=1}^\infty\frac{b_F(p^k)}{p^{ks}}$.

\begin{remark}\label{remark:convergence_partial_prod}
For any $n\geq 1$, $F_n(s)=\prod_{m=1}^n F_{p_m}(s)$ is well defined as a Dirichlet series (and Euler product) absolutely convergent for $\sigma>\theta=\theta_F$ by \ref{hp:EP}. 
\end{remark}

\begin{lemma}\label{lemma:uniform_convergence}
Let $F(s)$ be a function satisfying \ref{hp:DS} and \ref{hp:EP}, and $\sigma_0>1$. Then $F_n(s)$ and $\log F_n(s)$ converge uniformly for $\sigma\geq \sigma_0$ respectively toward $F(s)$ and $\log F(s)$.
\end{lemma}
\begin{proof}
The proof is a simple consequence of the absolute convergence of the Dirichlet series \ref{hp:DS} and of the Euler product \ref{hp:EP} for $\sigma>1$.
\end{proof}

Let $N$ be a positive integer and $F_1(s),\ldots,F_N(s)$ orthogonal functions satisfying \ref{hp:DS} and \ref{hp:EP}. For $\boldsymbol\theta\in[0,1]^n$, we define
$$\boldsymbol{F}_n(\sigma,\boldsymbol\theta)= (F_{1,n}(\sigma,\boldsymbol\theta),\ldots,F_{N,n}(\sigma,\boldsymbol\theta))$$
and
$$\boldsymbol{\log F}_n(\sigma,\boldsymbol\theta)= (\log F_{1,n}(\sigma,\boldsymbol\theta),\ldots,\log F_{N,n}(\sigma,\boldsymbol\theta)).$$
To these functions we attach some distribution functions, namely for any Borel set $\boldsymbol{E}\subseteq\C^N$, $j,l\in\{1,\ldots, N\}$, $j\neq l$ and $\sigma>1$ we set
\begin{equation}\label{eq:def_lambda_j}
\bb\lambda_{\sigma,n;j}(\boldsymbol{E})=\int_{\bb{W}_{\bb{\log F}_n}(\sigma,\boldsymbol{E})}\abs{\frac{F'_{j,n}}{F_{j,n}}(\sigma,\boldsymbol\theta)}^2 d\boldsymbol\theta
\end{equation}
and
\begin{equation}\label{eq:def_lambda_j_l_tau}
\bb\lambda_{\sigma,n;j,l;\tau}(\boldsymbol{E})=\int_{\bb{W}_{\bb{\log F}_n}(\sigma,\boldsymbol{E})}\abs{\frac{F'_{j,n}}{F_{j,n}}(\sigma,\boldsymbol\theta)+\tau \frac{F'_{l,n}}{F_{l,n}}(\sigma,\boldsymbol\theta)}^2 d\boldsymbol\theta,
\end{equation}
where $\bb{W}_{\bb{\log F}_n}(\sigma,\bb{E}) = \{\boldsymbol\theta\in [0,1)^n \mid \bb{\log F}_n(\sigma,\boldsymbol\theta)\in \bb{E}\}$, and $\tau=\pm1,\pm i$.

We say that a distribution function $\mu$ on $\C^n$ is \emph{absolutely continuous} (with respect to the Lebesgue measure $meas$), if for every Borel set $E\subseteq \C^n$, $meas(E)=0$ implies $\mu(E)=0$ (cf. Bogachev \cite[Definition 3.2.1]{bogachev}). By Radon--Nikodym's Theorem (Bogachev \cite[Theorem 3.2.2]{bogachev}) this holds if and only if there exists a Lebesgue integrable function $G_\mu:\C^n\rightarrow \R_{\geq 0}$ such that
$$\mu(E)=\int_E G_\mu(\boldsymbol{x})d\boldsymbol{x}$$
for any Borel set $E\subseteq\C^n$. We call $G_\mu(\boldsymbol{x})$ the \emph{density} of $\mu$.

As a sufficient condition for absolute continuity we write here, for easier reference, the following result (cf. Borchsenius and Jessen \cite[\S6]{borchseniusjessen} and Bogachev \cite[\S3.8]{bogachev}).
\begin{lemma}\label{lemma:abs_cont}
Let $\mu$ be a distribution function on $\C^n$ and let $\widehat{\mu}$ be its Fourier transform. Suppose $\int_{\C^n}\norma{\boldsymbol{y}}^q|\widehat{\mu}(\boldsymbol{y})|d\boldsymbol{y}<\infty$ for some integer $q\geq 0$, then $\mu$ is absolutely continuous with density $G_\mu(\boldsymbol{x})\in\mathcal{C}^q(\C^n)$ determined by the Fourier inversion formula
$$G_\mu(\boldsymbol{x})=\frac{1}{(2\pi)^{2n}}\int_{\C^n} \exp{-i\langle \boldsymbol{x},\boldsymbol{y}\rangle}\widehat\mu(\boldsymbol{y})d\boldsymbol{y}.$$
\end{lemma}

On the distribution functions defined above we have the following result.

\begin{theorem}\label{theorem:distribution_functions}
Let $N$ be a positive integer and $F_1(s),\ldots,F_N(s)$ orthogonal functions satisfying \ref{hp:DS} and \ref{hp:EP}. Then there exists $n_0\geq 1$ such that the distribution functions $\bb\lambda_{\sigma,n;j}$ and $\bb\lambda_{\sigma,n;j,l;\tau}$ are absolutely continuous with continuous densities $G_{\sigma,n;j}(\bb{x})$ and $G_{\sigma,n;j,l;\tau}(\bb{x})$ for every $n\geq n_0$, $\sigma\geq 1$, $j,l\in\{1,\ldots,N\}$, $j\neq l$ and $\tau=\pm1,\pm i$; more generally for any $k\geq 0$ there exists $n_k\geq 1$ such that $G_{\sigma,n;j}(\bb{x}),\,G_{\sigma,n;j,l;\tau}(\bb{x})\in\mathcal{C}^k(\C^N)$ for every $n\geq n_k$, $\sigma\geq 1$.\\
Moreover, $\bb\lambda_{\sigma,n;j}$ and $\bb\lambda_{\sigma,n;j,l;\tau}$ converge weakly to some distribution functions $\bb\lambda_{\sigma;j}$ and $\bb\lambda_{\sigma;j,l;\tau}$ as $n\rightarrow \infty$, which are absolutely continuous with densities $G_{\sigma;j}(\bb{x}),\,G_{\sigma;j,l;\tau}(\bb{x})\in\mathcal{C}^\infty(\C^N)$ for every $\sigma\geq 1$, $j,l\in\{1,\ldots,N\}$, $j\neq l$ and $\tau=\pm1,\pm i$. The functions  $G_{\sigma,n;j}(\bb{x})$ and $G_{\sigma,n;j,l;\tau}(\bb{x})$ and their partial derivatives converge uniformly for $\bb{x}\in\C^n$ and $1\leq \sigma\leq M$ towards $G_{\sigma;j}(\bb{x})$ and $G_{\sigma;j,l;\tau}(\bb{x})$ and their partial derivatives as $n\rightarrow\infty$ for every $M>1$.
\end{theorem}
\begin{proof}
The proof is an adaptation of Theorem 5 of Borchsenius and Jessen \cite{borchseniusjessen} (see also Lee \cite[pp. 1827--1830]{lee}). We prove it just for $\bb\lambda_{\sigma,n;j}$ since the proof for the others is completely similar.

We compute the Fourier transform of the distribution functions $\bb\lambda_{\sigma,n;j}$ and we get
\begin{equation}\label{eq:hat_lambda_j}
\widehat{\bb\lambda_{\sigma,n; j}}(\bb{y}) = \int_{[0,1]^n}\exp{i\sum_{h=1}^N \re(\log F_{h,n}(\sigma,\boldsymbol\theta)\conj{y_h})}\abs{\frac{F'_{j,n}}{F_{j,n}}(\sigma,\boldsymbol\theta)}^2 d\boldsymbol\theta,
\end{equation}
for any $\bb{y}=(y_1,\ldots,y_N)\in\C^N$. By Lemma \ref{lemma:abs_cont}, to prove the first part it is sufficient to show that for every $k\geq 0$ there exists $n_k$ such that, for any $M>1$, $\norma{\bb{y}}^k\widehat{\bb\lambda_{\sigma,n;j}}(\bb{y})$ is Lebesgue integrable for every $n\geq n_k$ and $1\leq\sigma\leq M$. We recall that by \ref{hp:EP} there exist $K_{F_j}$ and $\theta_{F_j}<\frac12$ such that $|b_{F_j}(p^n)|\leq K_{F_j}p^{n\theta_{F_j}}$ for every prime $p$ and $k\geq 1$, $j=1,\ldots,N$. Then we have
\begin{equation}\label{eq:bound_trasf_fourier_compact}
\abs{\widehat{\bb\lambda_{\sigma,n;j}}(\bb{y})}\leq \sup_{\sigma>1}\abs{\frac{F'_{j,n}}{F_{j,n}}(\sigma,\boldsymbol\theta)}^2\leq \sum_{m=1}^n \log^2 p_m\sum_{k=1}^\infty \frac{|b_{F_j}(p_m^{k})|^2}{p_m^{2k\sigma}}\leq K_{F_j}^2 \sum_{p}\frac{\log^2p}{p^{2(\sigma-\theta_{F_j})}}<\infty
\end{equation}
for every $n\geq 1$ and $\sigma\geq 1$. Hence it is sufficient to show that there exist constants $C_k>0$ and $n_k\geq 1$ such that for any $M>1$ we have
$$\abs{\widehat{\bb\lambda_{\sigma,n;j}}(\bb{y})}\leq C_k \norma{\bb{y}}^{-5/2-k}\quad\hbox{as }\norma{\bb{y}}\rightarrow \infty$$
for every $n\geq n_k$ and $1\leq \sigma\leq M$. To prove this, note that we can write (cf. Borchsenius and Jessen \cite[(47)]{borchseniusjessen} and Lee \cite[(3.24)]{lee})
\begin{equation}\label{eq:lambda_expanded}
\widehat{\bb\lambda_{\sigma,n;j}}(\bb{y})=\sum_{m=1}^n K_{2,j}(p_m,\bb{y}) \prod_{\substack{\ell=1\\ \ell\neq m}}^n K_{0,j}(p_\ell,\bb{y})+\sum_{\substack{m,k=1\\m\neq k}}^n K_{1,j}(p_m,\bb{y})\conj{K_{1,j}(p_k,-\bb{y})}\prod_{\substack{\ell =1\\ \ell\neq m,k}}^n K_{0,j}(p_\ell,\bb{y}),
\end{equation}
where, for any prime $p$ and $j\in\{1,\ldots,N\}$, we take
\begin{spliteq}\label{eq:kappas}
K_{0,j}(p,\bb{y}) &= \int_0^1 \exp{i\sum_{h=1}^N\re\left(\log F_{h,p}\left(\sigma+i\frac{2\pi\theta}{\log p}\right)\conj{y_h}\right)}d\theta,\\
 K_{1,j}(p,\bb{y}) &= \int_0^1 \exp{i\sum_{h=1}^N\re\left(\log F_{h,p}\left(\sigma+i\frac{2\pi\theta}{\log p}\right)\conj{y_h}\right)}\frac{F'_{j,p}}{F_{j,p}}\left(\sigma+i\frac{2\pi\theta}{\log p}\right)d\theta,\\
\hbox{and } K_{2,j}(p,\bb{y}) &= \int_0^1 \exp{i\sum_{h=1}^N\re\left(\log F_{h,p}\left(\sigma+i\frac{2\pi\theta}{\log p}\right)\conj{y_h}\right)}\abs{\frac{F'_{j,p}}{F_{j,p}}\left(\sigma+i\frac{2\pi\theta}{\log p}\right)}^2 d\theta.
\end{spliteq}
Hence, we just need to estimate the functions defined in \eqref{eq:kappas}.

For all primes $p$ and $j\in\{1,\ldots,N\}$ we clearly have
\begin{equation}\label{eq:kappa_0_trivial}
|K_{0,j}(p,\bb{y})|\leq 1.
\end{equation}
On the other hand, by Theorem \ref{theorem:bound_trasf_fourier_primes} there are a positive constant $A$ and infinitely many primes $p$ such that 
\begin{equation}\label{eq:kappa_0_theorem}
\abs{K_{0,j}(p,\bb{y})}\leq \frac{A}{\sqrt{\norma{\bb{y}}}}p^{\frac{\sigma}{2}}
\end{equation}
for every $\sigma\geq 1$, $\bb{y}\neq \bb0$ and $j\in\{1,\ldots,N\}$. Thus, putting together \eqref{eq:kappa_0_trivial} and \eqref{eq:kappa_0_theorem} we obtain that for any fixed integer $q\geq 1$ there exists $m_q$ such that
\begin{equation}\label{eq:kappa_0_prod}
\prod_{\substack{\ell =1\\ \ell\neq m,k}}^n \abs{K_{0,j}(p_\ell,\bb{y})}\leq \left[\frac{A}{\sqrt{\norma{\bb{y}}}}p_{m_q}^{\frac{\sigma}{2}}\right]^q
\end{equation}
for every $m,k\leq n$, $n\geq m_q$, $\sigma\geq 1$, $\bb{y}\neq \bb0$ and $j\in\{1,\ldots,N\}$.\\
Since we shall need it later, we also note that from the fact that $|\exp{it}-1-it|\leq \frac{t^2}{2}$ and by \ref{hp:EP}, for every prime $p$ we get (cf. Borchsenius and Jessen \cite[(50)]{borchseniusjessen} and Lee \cite[p. 1830]{lee})
\begin{equation}\label{eq:kappa_0_1}
|K_{0,j}(p,\bb{y})-1| \leq \frac{\norma{\bb{y}}^2}{2}\left(\sum_{h=1}^N K_{F_j}^2\right)\frac{1}{p^{2(\sigma-\max_h\theta_{F_h})}}.
\end{equation}

For $K_{1,j}(p,\bb{y})$, using the fact that $|\exp{it}-1|\leq |t|$ and \ref{hp:EP}, we obtain for any $\sigma\geq1$ and any prime $p$ (cf. Borchsenius and Jessen \cite[(52)]{borchseniusjessen} and Lee \cite[(3.27)]{lee})
\begin{equation}\label{eq:kappa_1}
|K_{1,j}(p,\bb{y})| \leq \norma{\bb{y}} K_{F_j}\sqrt{\sum_{h=1}^N K_{F_h}^2}\frac{\log p}{p^{2(\sigma-\max_h\theta_{F_h})}}.
\end{equation}

Finally, for any prime $p$, $\sigma\geq1$ and $j\in\{1,\ldots,N\}$, we simply have (cf. Borchsenius and Jessen \cite[(53)]{borchseniusjessen} and Lee \cite[(3.26)]{lee})
\begin{equation}\label{eq:kappa_2}
|K_{2,j}(p,\bb{y})| \leq \int_0^1 \abs{\frac{F'_{j,p}}{F_{j,p}}\left(\sigma+i\frac{2\pi\theta}{\log p}\right)}^2 d\theta \stackrel{\hbox{\ref{hp:EP}}}{\leq} K_{F_j}^2\frac{\log^2 p}{p^{2(\sigma-\theta_{F_j})}}.
\end{equation}

Putting together \eqref{eq:kappa_0_prod}, \eqref{eq:kappa_1} and \eqref{eq:kappa_2} into \eqref{eq:lambda_expanded}, for any fixed $M>1$, $j\in\{1,\ldots,N\}$ and $q\geq 0$ we get 
\begin{spliteq*}
\abs{\widehat{\lambda_{\sigma,n;j}}(\bb{y})}\leq & K_{F_j}^2 A^q \norma{\bb{y}}^{-\frac{q}{2}}p_{m_q}^{\frac{q\sigma}{2}}\sum_{m=1}^n \frac{\log^2 p_m}{p_m^{2(\sigma-\theta_{F_j})}}\\
&\quad+K_{F_j}^2\left(\sum_{h=1}^N K_{F_h}^2\right) A^q \norma{\bb{y}}^{2-\frac{q}{2}}p_{m_q}^{\frac{q\sigma}{2}}\left(\sum_{m=1}^n \frac{\log p_m}{p_m^{2(\sigma-\max_h\theta_{F_h})}}\right)^2
\end{spliteq*}
for any $n\geq m_q$, $\sigma\geq 1$ and $\bb{y}\neq \bb0$. Choosing $q=9+2k$, $n_k = m_{9+2k}$ and 
$$C_k=\left(\sum_{h=1}^N K_{F_h}^2\right) A^{9+2k} p_{n_k}^{(9+2k)M/2}\left(1+\left(\sum_{h=1}^N K_{F_h}^2\right)^2\sum_{p} \frac{\log p}{p^{2(\sigma-\max_h\theta_{F_h})}}\right)\sum_p \frac{\log p}{p^{2(\sigma-\max_h\theta_{F_h})}}$$
we have 
\begin{equation}\label{eq:dominated}
\abs{\widehat{\bb\lambda_{\sigma,n;j}}(\bb{y})}\leq C_k \norma{\bb{y}}^{-\frac{5}{2}-k}\quad\hbox{when }\norma{\bb{y}}\geq 1,
\end{equation}
for every $n\geq n_k=m_{9+2k}$, $1\leq\sigma\leq M$ and $j\in\{1,\ldots,N\}$. Hence the distribution functions $\bb\lambda_{\sigma,n;j}$, $j=1,\ldots,N$, are absolutely continuous with continuous density for every $n\geq n_0$ and every $1\leq\sigma\leq M$, while $G_{\sigma,n;j}(\bb{x})\in\mathcal{C}^k(\C^N)$ for every $n\geq n_k$ and every $1\leq\sigma\leq M$. Moreover, since $n_k$ doesn't depend on $M$ and since $M$ is arbitrary, it follows that $\bb\lambda_{\sigma,n;j}$, $j=1,\ldots,N$, are absolutely continuous with continuous density for every $n\geq n_0$ and every $\sigma\geq 1$, while $G_{\sigma,n;j}(\bb{x})\in\mathcal{C}^k(\C^N)$ for every $j\in\{1,\ldots,N\}$, $n\geq n_k$ and $\sigma\geq 1$.

On the other hand, by \eqref{eq:bound_trasf_fourier_compact}, \eqref{eq:lambda_expanded}, \eqref{eq:kappa_0_trivial}, \eqref{eq:kappa_0_1}, \eqref{eq:kappa_1}, and \eqref{eq:kappa_2}, we have (cf. Borchsenius and Jessen \cite[(60)]{borchseniusjessen} and Lee \cite[p. 1830]{lee})
\begin{equation*}
\abs{\widehat{\bb\lambda_{\sigma,n+1;j}}(\bb{y})-\widehat{\bb\lambda_{\sigma,n;j}}(\bb{y})}\ll \norma{\bb{y}}^{2}\frac{\log p_{n+1}}{p_{n+1}^{2(\sigma-\max_h\theta_{F_h})}}
\end{equation*}
for every $n\geq 1$, $\sigma\geq1$ and $j\in\{1,\ldots,N\}$. By the triangle inequality we thus get
\begin{equation}\label{eq:cauchy}
\abs{\widehat{\bb\lambda_{\sigma,n+k;j}}(\bb{y})-\widehat{\bb\lambda_{\sigma,n;j}}(\bb{y})}\ll \norma{\bb{y}}^{2}\sum_{m=n+1}^{n+k}\frac{\log p_{m}}{p_{m}^{2(\sigma-\max_h\theta_{F_h})}}\leq \norma{\bb{y}}^{2}\sum_{m=n+1}^{\infty}\frac{\log p_{m}}{p_{m}^{2(\sigma-\max_h\theta_{F_h})}} 
\end{equation}
for every $n,\,k\geq 1$ and $\sigma\geq1$. Hence, by Cauchy's convergence criterion, there exist the limit functions
$$\widehat{\bb\lambda_{\sigma;j}}(\bb{y}) = \lim_{n\rightarrow \infty} \widehat{\bb\lambda_{\sigma,n;j}}(\bb{y}),\quad j=1,\ldots,N,$$
and by \eqref{eq:cauchy} it is clear that the convergence is uniform in $\norma{\bb{y}}\leq a$, for every $a>0$. Therefore, by L\'evy's convergence theorem (cf. Bogachev \cite[Theorem 8.8.1]{bogachev}), we have that $\widehat{\bb\lambda_{\sigma;j}}(\bb{y})$ is the Fourier transform of some distribution function $\bb\lambda_{\sigma;j}$ and $\bb\lambda_{\sigma,n;j}\rightarrow \bb\lambda_{\sigma;j}$ weakly as $n\rightarrow \infty$, for $j=1,\ldots,N$. Moreover by \eqref{eq:dominated} we have that we may apply the dominated convergence theorem and thus $\bb\lambda_{\sigma;j}$ are absolutely continuous for every $\sigma\geq1$ and $j\in\{1,\ldots, N\}$, with density $G_{\sigma;j}(\bb{x})\in \mathcal{C}^\infty(\C)$ (for the arbitrariness of $M$ and $k$). Moreover, since $G_{\sigma,n;j}(\bb{x})$ and $G_{\sigma;j}(\bb{x})$ are determined by the inverse Fourier transform (see Lemma \ref{lemma:abs_cont}), the dominated convergence theorem yields that $G_{\sigma,n;j}(\bb{x})$ and their partial derivatives converge uniformly for $\bb{x}\in\C^n$ and $1\leq \sigma\leq M$ toward $G_{\sigma;j}(\bb{x})$ and their partial derivatives for every $j\in\{1,\ldots,N\}$.
\end{proof}

\begin{theorem}\label{theorem:majorant_densities}
For any $\alpha>0$ and $q\geq 0$ the densities $G_{\sigma;j}(\bb{x})$ and $G_{\sigma,n;j}(\bb{x})$, $n\geq n_q$, together with their partial derivatives of order $\leq q$, have a majorant of the form $K_q \exp{-\alpha\norma{\bb{x}}^2}$ for every $\sigma\geq 1$, $j,l\in\{1,\ldots,N\}$, $j\neq l$ and $\tau=\pm1,\pm i$.
\end{theorem}
\begin{proof} Straightforward adaptation of Theorems 6 and 9 of Borchsenius and Jessen \cite{borchseniusjessen}.
\end{proof}

\begin{theorem}\label{theorem:continuity_sigma}
The distribution functions $\bb\lambda_{\sigma;j}$, $\bb\lambda_{\sigma;j,l;\tau}$, $\bb\lambda_{\sigma,n;j}$ and $\bb\lambda_{\sigma,n;j,l;\tau}$, for $n\geq n_0$, depend continuously on $\sigma$, and their densities $G_{\sigma;j}(\bb{x})$, $G_{\sigma;j,l;\tau}(\bb{x})$, $G_{\sigma,n;j}(\bb{x})$ and $G_{\sigma,n;j,l;\tau}(\bb{x})$, together with their partial derivatives of order $\leq q$ if $n\geq n_q$, are continuous in $\sigma$ for every $\sigma\geq1$, $j,l\in\{1,\ldots,N\}$, $j\neq l$ and $\tau=\pm1,\pm i$.
\end{theorem}
\begin{proof} As in Theorem 9 of Borchsenius and Jessen \cite{borchseniusjessen} e result follows from \eqref{eq:dominated}, \eqref{eq:cauchy} and the Fourier inversion formula.
\end{proof}

\begin{remark}\label{remark:distribution_functions_1_2}
As for Remark \ref{remark:convexity_1_2}, note that Theorems \ref{theorem:distribution_functions}, \ref{theorem:majorant_densities} and \ref{theorem:continuity_sigma} hold for $\sigma>1$ because $\max_j \theta_{F_j}<\frac{1}{2}$ by \ref{hp:EP}. Therefore if we had that $\max_j \theta_{F_j}<\frac{\kappa}{2}$ for some $0<\kappa<1$ we would immediately have that \eqref{eq:bound_trasf_fourier_primes} holds for every $\sigma>\kappa$.
\end{remark}

\section{Zeros of sums of two Euler products}\label{section:proof_main1}

Let $F_1(s)$ and $F_2(s)$ be functions satisfying \ref{hp:DS} and \ref{hp:EP}, and $c_1,c_2\in\C\setminus\{0\}$. We then set
$$L(s)=c_1F_1(s)+c_2F_2(s).$$
To study the distribution of the zeros of $L(s)$ for $\sigma>1$, we note that, since $F_1(s)F_2(s)\neq 0$ for $\sigma>1$, then
$$L(s)=0 \quad \Leftrightarrow \quad \log\left(\frac{F_1(s)}{F_2(s)}\right)=\log\left(-\frac{c_2}{c_1}\right).$$
This idea was used by Gonek \cite{gonek}, and later by Bombieri and Mueller \cite{bombierimueller} and Bombieri and Ghosh \cite{bombierighosh}. Moreover, if $F_1(s)$ and $F_2(s)$ are orthogonal, then it is easy to show that $\frac{F_1}{F_2}(s)$ satisfies \ref{hp:DS}, \ref{hp:EP} and, if we write $\frac{F_1}{F_2}(s)=\sum_{n=1}^{\infty} a(n)n^{-s}$,
\begin{equation}\label{eq:normality}
\sum_{p\leq x} \frac{|a(p)|^2}{p} = (\kappa+o(1))\log\log x,\quad x\rightarrow \infty,
\end{equation}
for some constant $\kappa>0$.
Therefore Theorem \ref{theorem:main1} follows immediately from the following more general result on the value distribution of the logarithm of an Euler product.

\begin{theorem}\label{theorem:density_2_EP}
Let $F(s)$ be a function satisfying \ref{hp:DS}, \ref{hp:EP} and \eqref{eq:normality}, and $c\in\C$. Then the Dirichlet series $\log F(s)-c$ has no isolated vertical lines containing zeros in the half-plane $\sigma>1$.
\end{theorem}

\begin{proof}The first part of the proof is similar to Borchsenius and Jessen's application of Theorems 5 to 9 in \cite{borchseniusjessen} to the Riemann zeta function \cite[Theorems 11 and 13]{borchseniusjessen}.

For every $n\geq 1$ consider the Dirichlet series $\log F_n(s)$, which are absolutely convergent for $\sigma>\theta_F$ by Remark \ref{remark:convergence_partial_prod}. Let $\nu_{\sigma,n}$ be, for every $\sigma>\theta_F$, the \emph{asymptotic distribution function} of $\log F_n(s)$ with respect to $\abs{\frac{F'_n}{F_n}(s)}^2$ defined for any Borel set $E\subseteq\C$ by (cf. Borchsenius and Jessen \cite[\S7]{borchseniusjessen})
$$\nu_{\sigma,n}(E)=\lim_{T_2-T_1\rightarrow \infty} \frac{1}{T_2-T_1}\int_{V_{\log F_n}(\sigma,T_1,T_2,E)}\abs{\frac{F'_n}{F_n}(s)}^2 dt,$$
where $V_{\log F_n}(\sigma,T_1,T_2,E)=\{t\in(T_1,T_2)\mid \log F_n(\sigma+it)\in E\}$. For $\sigma\geq 1$, we compute its Fourier transform and, by Kronecker-Weyl's theorem (see Karatsuba and Voronin \cite[\S A.8]{karatsubavoronin}) we get (cf. Borchsenius and Jessen \cite[p. 160]{borchseniusjessen} or Lee \cite[p. 1819]{lee})
$$\widehat{\nu_{\sigma,n}}(y) = \int_{[0,1]^n}\exp{i\re(\log F_n(\sigma,\boldsymbol\theta)\conj{y})}\abs{\frac{F'_n}{F_n}(\sigma,\boldsymbol\theta)}^2d\boldsymbol\theta \stackrel{\eqref{eq:hat_lambda_j}}{=} \widehat{\bb\lambda_{\sigma,n; 1}}(y),$$
with $N=1$. For simplicity we write $\lambda_{\sigma,n} = \bb\lambda_{\sigma,n; 1}$. By the uniqueness of the Fourier transform (cf. Bogachev \cite[Proposition 3.8.6]{bogachev}) we have that $\nu_{\sigma,n}=\lambda_{\sigma,n}$ as distribution functions for every $\sigma\geq1$ and $n\geq 1$.

By Theorem \ref{theorem:distribution_functions} we know that $\nu_{\sigma,n}=\lambda_{\sigma,n}$ is absolutely continuous for $n\geq n_0$ with density $G_{\sigma,n}(x)$ which is a continuous function of both $\sigma$ and $x$ (see Theorem \ref{theorem:continuity_sigma}). Hence for any $n\geq n_0$, $x\in\C$ and $\sigma>\theta_F$ we have that the \emph{Jensen function} $\varphi_{\log F_n-x}(\sigma)$ (see Jessen and Tornehave \cite[Theorem 5]{jessentornehave}) is twice differentiable with continuous second derivative (cf. Borchsenius and Jessen \cite[\S9]{borchseniusjessen})
\begin{equation}\label{eq:phi_2_log_F_n}
\varphi_{\log F_n-x}''(\sigma) = 2\pi G_{\sigma,n}(x).
\end{equation}
On the other hand, for any $1<\sigma_1<\sigma_2$, by the uniform convergence of $\log F_n(s)$ of Lemma \ref{lemma:uniform_convergence} and by Jessen and Tornehave \cite[Theorem 6]{jessentornehave}, we have that
\begin{equation}\label{eq:phi_log_F_n}
\varphi_{\log F_n -x}(\sigma)\rightarrow \varphi_{\log F -x}(\sigma)\qquad\hbox{ as }n\rightarrow \infty
\end{equation}
uniformly for $\sigma_1\leq\sigma\leq \sigma_2$. Moreover, by Theorem \ref{theorem:distribution_functions}, also $G_{\sigma,n}(x)$ converges uniformly for $\sigma_1\leq\sigma\leq \sigma_2$ toward $G_{\sigma}(x)$, which is continuous in both $\sigma$ and $x$. Then, by \eqref{eq:phi_2_log_F_n}, \eqref{eq:phi_log_F_n}, the convexity of $\varphi_{\log F_n -x}$ and Rudin \cite[Theorem 7.17]{rudin2} we obtain that for any $x\in\C$ the Jensen function $\varphi_{\log F-x}(\sigma)$ is twice differentiable with continuous second derivative
$$\varphi_{\log F-x}''(\sigma)=2\pi G_\sigma(x).$$

We fix an arbitrary $c\in\C$ and we note the following: suppose that $\varphi''_{\log F-c}(\sigma_0)>0$ for some $\sigma_0>1$. Then, by continuity, there exists $\eps_0>0$ such that $\varphi''_{\log F-c}(\sigma)>0$ for every $\sigma \in (\sigma_0-\eps_0,\sigma_0+\eps_0)$. Then, for any $0<\eps<\eps_0$, by Theorem 31 of Jessen and Tornehave \cite{jessentornehave} and the mean value theorem, we have
\begin{spliteq*}
\lim_{T_2-T_1\rightarrow\infty} \frac{N_{\log F-c}(\sigma_0-\eps,\sigma_0+\eps,T_1,T_2)}{T_2-T_1} &= \frac{1}{2\pi}\left(\varphi_{\log F-c}'(\sigma_0+\eps)-\varphi_{\log F-c}'(\sigma_0-\eps)\right)\\
&=\frac{\eps}{2\pi}\varphi_{\log F-c}''(\sigma_\eps)>0,
\end{spliteq*}
for some $\sigma_\eps\in(\sigma_0-\eps,\sigma_0+\eps)$, i.e. there are infinitely many zeros with real part $\sigma\in (\sigma_0-\eps,\sigma_0+\eps)$. This means, by letting $\eps\rightarrow0^+$, that $\sigma_0$ is the limit point of the real parts of some zeros of $\log F(s)-c$ (or $\sigma_0$ is itself a zero).

Now, suppose that there exists $\rho_0=\beta_0+i\gamma_0$ with $\beta_0>1$ such that $\log F(\rho_0)-c=0$. If we suppose that $\varphi''_{\log F-c}(\beta_0)>0$, then $\sigma=\beta_0$ cannot be an isolated vertical line containing zeros since $\beta_0$ is the limit point of the real parts of some zeros.\\
Suppose otherwise that $\varphi''_{\log F-c}(\widetilde\sigma)=0$, and for any $\delta>0$ consider the intervals $I_\delta^+=(\widetilde\sigma,\widetilde\sigma+\delta)$ and $I_\delta^-=(\widetilde\sigma-\delta,\widetilde\sigma)$. Note that in general, if $\varphi''_{\log F-c}(\sigma)=0$ for every $\sigma\in(\sigma_1,\sigma_2)$, for some $1<\sigma_1<\sigma_2$, then Theorem 31 of Jessen and Tornehave \cite{jessentornehave} and the mean value theorem imply that $\log F(s)-c$ has no zeros for $\sigma_1<\sigma<\sigma_2$. Therefore, in at least one between $I_\delta^+$ and $I_\delta^-$ there are infinitely many $\sigma$ such that $\varphi''_{\log F-c}(\sigma)>0$, for any $\delta>0$, by almost periodicity. Hence, letting $\delta\rightarrow0$, we see that there exists a sequence $\{\sigma_\delta\}_\delta$ such that $\varphi''_{\log F-c}(\sigma_\delta)>0$ and $\sigma_\delta\rightarrow\beta_0$. Since every $\sigma_\delta$ is the limit point of the real parts of some zeros, we conclude that also $\beta_0$ is the limit point of the real parts of some zeros.
\end{proof}
\section{\emph{c}-values of sums of at least three Euler products}\label{section:proof_main2}

We first state the following simple result which is a generalization of Lemma 2.4 of Lee \cite{lee}.
\begin{lemma}\label{lemma:bound_int_finite_prod}
Let $F(s)$ be a complex function satisfying \ref{hp:DS}, \ref{hp:EP} and \ref{hp:RC}, $\sigma_0>1/2$ and $k$ a fixed positive integer. Then there exists a positive constant $A_k(\sigma_0)$ such that
$$\int_{[0,1]^n} |F_n(\sigma,\boldsymbol{\theta})|^{2k}d\boldsymbol{\theta}\leq A_k(\sigma_0)\quad\hbox{and}\quad\int_{[0,1]^n} |F'_n(\sigma,\boldsymbol{\theta})|^{2k}d\boldsymbol{\theta}\leq A_k(\sigma_0)$$
for every $n\geq 1$ and $\sigma\geq \sigma_0$.
\end{lemma}
\begin{proof} As in Lemma 2.4 of Lee \cite{lee} the proof follows from a bound of 
$$\mathcal{J}_k(z_1,\ldots,z_n,w_1,\ldots,w_n)=\int_{[0,1]^n} \prod_{j=1}^k F_n(\sigma+z_j,\bb\theta)\overline{F_n(\sigma+\overline{w_j},\bb\theta)}d\bb\theta$$
and Cauchy's integral formula on polydiscs. This bound may be obtained with the same computations as in Lemma 2.5 of Lee \cite{lee} by replacing the Ramanujan bound $|a(n)|\leq 1$ with the weaker Ramanujan conjecture $|a(n)|\ll_\eps n^\eps$, where we take $0<\eps<(2\sigma_0-1)/4$.
\end{proof}

We can now prove Theorem \ref{theorem:main2}.

\begin{proof}[Proof of Theorem \ref{theorem:main2}]
To handle this case we follow an idea of Lee \cite[\S3.2]{lee} and we use the distribution functions studied in Section \ref{section:distribution_functions}, similarly to what we have done in the previous section for $N=2$.

For every $n\geq 1$ we write
$$L_n(s) = \sum_{j=1}^N c_jF_{j,n}(s)\quad\hbox{and}\quad L_n(\sigma,\boldsymbol{\theta})=L_n(\sigma,\theta_1,\ldots,\theta_n)=\sum_{j=1}^N c_jF_{j,n}(s,\theta_1,\ldots,\theta_n).$$
Let $\nu_{\sigma,n}$ be the asymptotic distribution function of $L_n(s)$ with respect to $\abs{L_n'(s)}^2$ defined for any Borel set $E\subseteq\C$ by (cf. Borchsenius and Jessen \cite[\S7]{borchseniusjessen})
$$\nu_{\sigma,n}(E)=\lim_{T_2-T_1\rightarrow \infty} \frac{1}{T_2-T_1}\int_{V_{L_n}(\sigma,T_1,T_2,E)}\abs{L'_n(s)}^2 dt,$$
where $V_{L_n}(\sigma,T_1,T_2,E)=\{t\in(T_1,T_2)\mid L_n(\sigma+it)\in E\}$. As in Theorem \ref{theorem:density_2_EP}, by the Kronecker-Weyl's Theorem and the uniqueness of the Fourier transform, we have that $\nu_{\sigma,n} = \lambda_{\sigma,n}$, for any $n\geq 1$ and $\sigma\geq 1$, where $\lambda_{\sigma,n}$ is the distribution function of $L_n(s,\bb\theta)$ with respect to $\abs{L_n'(s,\bb\theta)}^2$, defined for every Borel set $E\subseteq\C$ by
$$\lambda_{\sigma,n}(E)=\int_{W_{L_n}(\sigma,E)}\abs{L'_n(\sigma,\boldsymbol\theta)}^2 d\boldsymbol\theta,$$
with $W_{L_n}(\sigma,E) = \{\boldsymbol\theta=(\theta_1,\ldots,\theta_n)\in [0,1)^n \mid L_n(\sigma,\boldsymbol\theta)\in E\}$. We want to show that there exists $\widetilde{n}\geq 1$ such that $\lambda_{\sigma,n}$, and hence $\nu_{\sigma,n}$, is absolutely continuous with continuous density, which we call $H_{\sigma,n}(x)$, for every $n\geq \widetilde{n}$ and $\sigma\geq 1$.

As in Lee \cite[p. 1830--1831]{lee}, we compute the Fourier transform of $\lambda_{\sigma,n}$ and, for $\sigma\geq 1$ and $n\geq n_0$, we get
\begin{spliteq*}
\widehat{\lambda_{\sigma,n}}(y) &= \sum_{j,l=1}^N \conj{c_j}c_l (2\pi)^N\int_{\R_{+}^N}\int_{\R^N}\exp{i\sum_{h=1}^N |c_h\conj{y}|r_h\sin(2\pi (\theta_h-\alpha_h))-2\pi i\theta_j+2\pi i\theta_l}\\
&\qquad\qquad \times r_jr_lG_{\sigma,n;j,l}(\bb{r}) \frac{dr_1}{r_1}\cdots \frac{dr_N}{r_N} d\theta_1\cdots d\theta_N,
\end{spliteq*}
where $\bb{r} = (\log r_1 +2\pi i\theta_1,\ldots,\log r_N +2\pi i\theta_N)$, $\alpha_h$ is determined by the argument of $c_h\conj{y}$, for $h=1,\ldots,N$, and
$$G_{\sigma,n;j,l}(\bb{x}) = \left\{\begin{array}{ll}G_{\sigma,n;j}(\bb{x}) & j=l\\ \dsum_{\tau=\pm 1,\pm i}\conj{\tau} G_{\sigma,n;j,l;\tau}(\bb{x}) & j\neq l\end{array}\right.$$
is defined from the densities of the distribution functions $\bb\lambda_{\sigma,n;j}$ and $\bb\lambda_{\sigma,n;j,l;\tau}$ of Section \ref{section:distribution_functions}.

For any $h\in\{1,\ldots,N\}$ and any $\eps>0$ let $A_{h,\eps} = \{\theta\in\R\mid |\theta-\alpha_h-m\pi|<\eps $ for some $m\in\Z\}$. Then we note that integrating by parts with respect to $r_h$, $h=1,\ldots,N$, and using the majorant  $K_N \exp{-\left[\sum_{h=1}^N \log^2 r_h + \theta_h^2\right]}$ of Theorem \ref{theorem:majorant_densities} for the partial derivatives up to order $N$ of the density $G_{\sigma,n;j,l}(\bb{r})$, for $n\geq n_N$ and $\sigma\geq 1$, we obtain (cf. Lee \cite[p. 1832]{lee})
\begin{spliteq}\label{eq:first_int}
\int_{\R\setminus A_{1,\eps}}\cdots\int_{\R\setminus A_{N,\eps}}\int_{\R_{+}^N}&\exp{i\re\left(\sum_{h=1}^N r_hc_h\conj{y}\exp{2\pi i \theta_h}\right)-2\pi i\theta_j+2\pi i\theta_l}r_jr_lG_{\sigma,n;j,l}(\bb{r}) \frac{dr_1}{r_1}\cdots \frac{dr_N}{r_N} d\theta_1\cdots d\theta_N \\
&\ll  \prod_{h=1}^N \int_{\R\setminus A_{h,\eps}} \frac{1}{|c_h\conj{y}|\sin(2\pi (\theta_h-\alpha_h))}\exp{-\theta_h^2}d\theta_h\\
&\ll \frac{1}{(\eps|y|)^N}
\end{spliteq}
for every $n\geq n_N$, $\sigma\geq 1$ and $y\neq 0$. Analogously, integrating by parts with respect to $\theta_h$, $h=1,\ldots,N$, using van der Corput's lemma for oscillatory integrals (see Titchmarsh \cite[Lemma 4.2]{titchmarsh}) on each interval $[\alpha_h+m_h/2-\eps,\alpha_h+m_h/2+\eps]$ with $\eps<\frac{1}{2}$, and the majorant $K_N \exp{-\left[\sum_{h=1}^N \log^2 r_h + \theta_h^2\right]}$ of Theorem \ref{theorem:majorant_densities} for the partial derivatives up to order $N$ of the density $G_{\sigma,n;j,l}(\bb{r})$, $n\geq n_N$ and $\sigma\geq 1$, we obtain  (cf. Lee \cite[p. 1832]{lee})
\begin{spliteq}\label{eq:second_int}
\int_{\R_{+}^N}\int_{A_{1,\eps}}\cdots\int_{A_{N,\eps}}&\exp{i\re\left(\sum_{h=1}^N r_hc_h\conj{y}\exp{2\pi i \theta_h}\right)-2\pi i\theta_j+2\pi i\theta_l}r_jr_lG_{\sigma,n;j,l}(\bb{r}) \frac{dr_1}{r_1}\cdots \frac{dr_N}{r_N} d\theta_1\cdots d\theta_N \\
&\ll  \prod_{h=1}^N \int_{\R_{>0}} \frac{1}{|c_h\conj{y}|}\exp{-\log^2 r_h}dr_h\\
&\ll \frac{1}{|y|^N},
\end{spliteq}
for every $n\geq n_N$, $\sigma\geq 1$, $|y|\geq \max_h\frac{1}{|c_h|}$ and $\eps>0$ sufficiently small. Fixing $\eps>0$ sufficiently small so that \eqref{eq:second_int} holds and putting together \eqref{eq:first_int} and \eqref{eq:second_int}, we obtain
\begin{equation}\label{eq:abs_cont}
|\widehat{\nu_{\sigma,n}}(y)|=|\widehat{\lambda_{\sigma,n}}(y)|\ll |y|^{-N} \ll |y|^{-3}
\end{equation}
since $N\geq 3$, for every $n\geq n_N$, $\sigma\geq 1$ and $|y|\geq \max\left( 1,\max_h |c_h|^{-1}\right)$. By Lemma \ref{lemma:abs_cont} we have thus proved that $\nu_{\sigma,n}$ is absolutely continuous for every $n\geq \widetilde{n}=n_N$ and $\sigma\geq 1$. Moreover, since $\nu_{\sigma,n}$ depends continuously on $\sigma$ (cf. Borchsenius and Jessen \cite[\S7]{borchseniusjessen}), we have that $\widehat{\nu_{\sigma,n}}$ is continuous in $\sigma$. Therefore \eqref{eq:abs_cont} and the Fourier inversion formula imply that $H_{\sigma,n}(x)$ is continuous in both $\sigma$ and $x$. Note that all implied constants in \eqref{eq:abs_cont} are independent of $n$.

Now we prove that the absolutely continuous distribution functions $\lambda_{\sigma,n}$ converge weakly as $n\rightarrow\infty$ toward the absolutely continuous distribution function $\lambda_{\sigma}$ with density $H_\sigma(x)$ which is continuous in both $\sigma$ and $x$. Moreover, we want to show that, for any $1<\sigma_1<\sigma_2$, $H_{\sigma,n}(x)$ converges uniformly for $\sigma_1\leq \sigma\leq \sigma_2$ toward $H_\sigma(x)$ as $n\rightarrow \infty$.

For this, note that
\begin{spliteq}\label{eq:L_n_1}
L_{n+1}(\sigma,\bb\theta,\theta_{n+1}) &= \sum_{j=1}^N c_j F_{j,n}(\sigma,\bb\theta)F_{j,p_{n+1}}\left(\sigma +i \frac{2\pi \theta_{n+1}}{\log p_{n+1}}\right)\\
&\stackrel{\hbox{\ref{hp:RC}}}{=} \sum_{j=1}^N c_j F_{j,n}(\sigma,\bb\theta)\left(1 + \frac{a_{F_j}(p_{n+1})}{p_{n+1}^ {\sigma}}\exp{2\pi i \theta_{n+1}} + O_\eps\left(\frac{1}{p_{n+1}^{2(\sigma-\eps)}}\right)\right)\\
&=L_n(\sigma,\bb\theta) + \frac{\exp{2\pi i \theta_{n+1}}}{p_{n+1}^ {\sigma}}\sum_{j=1}^N c_j a_{F_j}(p_{n+1}) F_{j,n}(\sigma,\bb\theta) + O_\eps\left(\frac{\sum_j |F_{j,n}|}{p_{n+1}^{2(\sigma-\eps)}}\right)
\end{spliteq}
for every $\sigma\geq 1$ and $0<\eps<\frac{1}{2}$. Similarly
\begin{spliteq*}
L_{n+1}'(\sigma,\bb\theta,\theta_{n+1}) =&L_n'(\sigma,\bb\theta) + \frac{\exp{2\pi i \theta_{n+1}}}{p_{n+1}^{\sigma}}\sum_{j=1}^N c_j a_{F_j}(p_{n+1}) \left[F_{j,n}'(\sigma,\bb\theta)-\log p_{n+1} F_{j,n}(\sigma,\bb\theta)\right]\\
&\qquad +O_\eps\left(\frac{\log p_{n+1}\sum_j |F_{j,n}|+|F_{j,n}'|}{p_{n+1}^{2(\sigma-\eps)}}\right)\\
\end{spliteq*}
for every $\sigma\geq 1$ and $0<\eps<\frac{1}{2}$. Hence we have (cf. Lee \cite[(3.20)]{lee})
\begin{spliteq}\label{eq:difference_3}
\widehat{\lambda_{\sigma,n+1}}(y)-&\widehat{\lambda_{\sigma,n}}(y)\\
&= \int_{[0,1]^{n+1}}\left[\exp{i\re\left(L_{n+1}(\sigma,\boldsymbol\theta,\theta_{n+1})\conj{y}\right)}-\exp{i\re\left(L_{n}(\sigma,\boldsymbol\theta)\conj{y}\right)}\right]\abs{L'_n(\sigma,\boldsymbol\theta)}^2 d\boldsymbol\theta d\theta_{n+1}\\
&\quad+\frac{2}{p_{n+1}^{\sigma}}\int_{[0,1]^{n+1}}\exp{i\re\left(L_{n+1}(\sigma,\boldsymbol\theta,\theta_{n+1})\conj{y}\right)}\re\Bigg(\conj{L_n'(\sigma,\bb\theta)} \exp{2\pi i \theta_{n+1}}\\
&\qquad \times\sum_{j=1}^N c_j a_{F_j}(p_{n+1}) \left(F_{j,n}'(\sigma,\bb\theta)-\log p_{n+1} F_{j,n}(\sigma,\bb\theta)\right)\Bigg)d\bb\theta d\theta_{n+1}\\
&\quad +O_\eps\left(\frac{\log p_{n+1}}{p_{n+1}^{2(\sigma-\eps)}}\int_{[0,1]^{n+1}}(1+\sum_j |F_{j,n}'|)\left(\sum_j |F_{j,n}|+|F_{j,n}'|\right)d\bb\theta d\theta_{n+1}\right)\\
&\quad +O_\eps\left(\frac{\log^2 p_{n+1}}{p_{n+1}^{4(\sigma-\eps)}}\int_{[0,1]^{n+1}}\left(\sum_j |F_{j,n}|+|F_{j,n}'|\right)^2 d\bb\theta d\theta_{n+1}\right).
\end{spliteq}
for every $\sigma\geq 1$ and $0<\eps<\frac{1}{2}$.\\
For the first term, using the fact that $|\exp{it}-1-it|\leq \frac{t^2}{2}$, we obtain (cf. Lee \cite[(3.22)]{lee})
$$\Bigg|\int_{0}^1 \left[\exp{i\re\left(L_{n+1}(\sigma,\boldsymbol\theta,\theta_{n+1})\conj{y}\right)}-\exp{i\re\left(L_{n}(\sigma,\boldsymbol\theta)\conj{y}\right)}\right]d\theta_{n+1}\Bigg| \ll_{\eps,a}\frac{\sum_j |F_{j,n}| + |F_{j,n}|^2}{p_{n+1}^{2(\sigma-\eps)}}$$
for $|y|\leq a$, $a>0$, $\sigma\geq 1$ and $0<\eps<\frac{1}{2}$. For the second term we get directly from \eqref{eq:L_n_1} and $|\exp{it}-1|\leq |t|$ that
$$\Bigg|\int_{0}^1 \exp{i\re\left(L_{n+1}(\sigma,\boldsymbol\theta,\theta_{n+1})\conj{y}\right)}\exp{\pm2\pi i\theta_{n+1}} d\theta_{n+1}\Bigg|\ll_{\eps,a}\frac{\sum_j |F_{j,n}|}{p_{n+1}^{(\sigma-\eps)}}$$
for $|y|\leq a$, $a>0$, $\sigma\geq 1$ and $0<\eps<\frac{1}{2}$. We fix $0<\eps<\frac{1}{2}$, then putting these together, by triangle inequality and Lemma \ref{lemma:bound_int_finite_prod} with $\sigma_0=1$, we get (cf. Lee \cite[p. 1826]{lee})
$$|\widehat{\lambda_{\sigma,n+1}}(y)-\widehat{\lambda_{\sigma,n}}(y)|\ll_{a,\eps} \frac{\log {p_{n+1}}}{p_{n+1}^{2(\sigma-\eps)}}$$
uniformly for $|y|\leq a$, $a>0$, and for every $\sigma \geq 1$. It follows, by triangle inequality, that for any $k>0$
\begin{equation}\label{eq:cauchy3}
\abs{\widehat{\lambda_{\sigma,n+k}}({y})-\widehat{\lambda_{\sigma,n}}({y})}\ll_{a,\eps} \sum_{m=n+1}^{n+k}\frac{\log p_{m}}{p_{m}^{2(\sigma-\eps)}}\leq \sum_{m=n+1}^{\infty}\frac{\log p_{m}}{p_{m}^{2(\sigma-\max_h\theta_{F_h})}} 
\end{equation}
for every $n,\,k\geq 1$ and $\sigma\geq 1$, uniformly for $|y|\leq a$, $a>0$. Hence, by Cauchy's convergence criterion, there exists the limit function
$$\widehat{\lambda_{\sigma}}({y}) = \lim_{n\rightarrow \infty} \widehat{\lambda_{\sigma,n}}({y})$$
and by \eqref{eq:cauchy} the convergence is uniform in $|y|\leq a$ for every $a>0$. Therefore, by L\'evy's convergence theorem, we have that $\widehat{\lambda_{\sigma}}({y})$ is the Fourier transform of some distribution function $\lambda_{\sigma}$, and $\lambda_{\sigma,n}\rightarrow \lambda_{\sigma}$ weakly as $n\rightarrow \infty$. Moreover, since the constants in \eqref{eq:abs_cont} are independent of $n$, we have that we may apply the dominated convergence theorem and thus $\lambda_{\sigma}$ is absolutely continuous for every $\sigma\geq 1$, with continuous (both in $\sigma$ and $x$) density $H_{\sigma}({x})$. Furthermore, since $H_{\sigma,n}(x)$ and $H_{\sigma}(x)$ are determined by the Fourier inversion formula (see Lemma \ref{lemma:abs_cont}), the uniform convergence of $\widehat{\lambda_{\sigma,n}}(y)\rightarrow \widehat{\lambda_{\sigma}}(y)$ and \eqref{eq:abs_cont} imply that $H_{\sigma,n}(x)$ converges, uniformly with respect to both $1\leq \sigma\leq M$, $M>1$, and $x\in\C$, toward $H_{\sigma}(x)$.

Now, similarly to Theorem \ref{theorem:density_2_EP}, for $n\geq \widetilde{n}$ and $c\in\C$ we have that the Jensen function $\varphi_{L_n-c}(\sigma)$ is twice differentiable with continuous second derivative (cf. Borchsenius and Jessen \cite[\S9]{borchseniusjessen})
\begin{equation}\label{eq:phi_2_L_n}
\varphi_{L_n-c}''(\sigma) = 2\pi H_{\sigma,n}(c).
\end{equation}
On the other hand, for any $1<\sigma_1<\sigma_2$, by the uniform convergence of $F_{j,n}(s)$, $j=1,\ldots,N$, of Lemma \ref{lemma:uniform_convergence} and by Jessen and Tornehave \cite[Theorem 6]{jessentornehave}, we have that
\begin{equation}\label{eq:phi_L_n}
\varphi_{L_n-c}(\sigma)\rightarrow \varphi_{L-c}(\sigma)\qquad\hbox{ as }n\rightarrow \infty
\end{equation}
uniformly for $\sigma_1\leq\sigma\leq \sigma_2$. By \eqref{eq:phi_2_L_n}, \eqref{eq:phi_L_n}, the convexity of $\varphi_{L_n-c}(\sigma)$ and Rudin \cite[Theorem 7.17]{rudin2} we obtain that the Jensen function $\varphi_{L}(\sigma)$ is twice differentiable with continuous second derivative
$$\varphi_{L-c}''(\sigma)=2\pi H_\sigma(c).$$
At this point, the same final argument of Theorem \ref{theorem:density_2_EP} yields the result.
\end{proof}

\section{Dirichlet series with vertical strips without zeros}\label{section:holes}

In this section we collect the proofs of Theorems \ref{theorem:one_hole}, \ref{theorem:holes} and \ref{theorem:func_eq_hole}.

\subsection{Proof of Theorem \ref{theorem:one_hole}}
We take $\bb{c}_0=(c_{0,1},\ldots,c_{0,N})\in\C^N\setminus\{\bb0\}$ such that $c_{0,1}a_1(1)+\cdots+c_{0,N}a_N(1)=0$. Since $L_{\bb{c}_0}(s)$ is not identically zero, then $\sigma^*(L_{\bb{c}_0})<+\infty$ and hence we fix 
$$\sigma_2>\sigma_1>\max(\sigma^*(L_{\bb{c}_0}),\max_{1\leq j\leq N}\sigma^*(F_j)).$$
Then, by definition of $\sigma^*(L_{\bb{c}_0})$ and Theorem 8 of Jessen and Tornehave \cite{jessentornehave}, there exists $\eps>0$ such that
$|L_{\bb{c}_0}(s)|>\eps$ for $\sigma_1\leq \sigma\leq \sigma_2$. Moreover, there exists $M>0$ such that $|F_j(s)|\leq M$ for $\sigma_1\leq \sigma\leq \sigma_2$. On the other hand, if we consider the hyperplanes $H(\sigma)=\{\bb{x}\in\C^N \mid L_{\bb{x}}(\sigma)=0\}$ we have 
$$\lim_{\sigma\rightarrow +\infty} dist(\bb{c}_0,H(\sigma))=\lim_{\sigma\rightarrow +\infty}\frac{|L_{\bb{c}_0}(\sigma)|}{\sqrt{\sum_j |F_j(\sigma)|^2}} = 0.$$
Therefore there exists $\beta>\sigma_2$ such that $dist(\bb{c}_0,H(\beta))<\frac{\eps}{4\sqrt{N}M}$.
Then for any $\bb0\neq \bb{z}\in B_{\frac{\eps}{2\sqrt{N}M}}(\bb{c}_0)\cap H(\beta)$ we have $L_{\bb{z}}(\beta)=0$ and, by the triangle and Cauchy-Schwartz inequalities,  
$$|L_{\bb{z}}(s)|\geq|L_{\bb{c}_0}(s)|-|L_{\bb{z}-\bb{c}_0}(s)|>\eps-\frac{\eps}{2}=\frac{\eps}{2}$$
for $1\leq \sigma^*(L_{\bb{c}_0})<\sigma_1\leq\sigma\leq\sigma_2<\beta\leq \sigma^*(L_{\bb{z}})$. Moreover, note that for every $0\neq r\in\C$ the Dirichlet series $L_{z\bb{c}}(s)$ has the same zeros and vertical strips without zeros of $L_{\bb{c}}(s)$.

\subsection{Proof of Theorem \ref{theorem:holes}}
We write $N=k+1\geq 2$. If $N=2$ then the result follows from Theorem \ref{theorem:one_hole}; so we suppose that $N\geq 3$.\\
Note $\bb{x}\in\C^N$ is such that $L_{\bb{x}}(\sigma)=0$ for some $\sigma>1$ if and only if $\bb{x}=(x_1,\ldots,x_N)$ belongs to the hyperplane
\begin{equation}\label{eq:hyper}
F_1(\sigma)x_1+\cdots + F_N(\sigma)x_N=0.
\end{equation}
If $\sigma>\max_{1\leq j\leq N}\:\sigma^*(F_j)=\widetilde\sigma_0$, then the space of solutions of \eqref{eq:hyper} has dimension $N-1\geq 2$ and is generated by
$$v_j^{(1)}(\sigma)=\left(-\frac{1}{F_1(\sigma)},0,\ldots,\frac{1}{F_j(\sigma)},\ldots,0\right),\qquad j=2,\ldots,N.$$
Moreover we define inductively for $h=2,\ldots,N-1$ the vectors
$$v_j^{(h)}(\sigma_1,\ldots,\sigma_h)= v_{j}^{(h-1)}(\sigma_1,\ldots,\sigma_{h-1})-\frac{L_{v_{j}^{(h-1)}(\sigma_1,\ldots,\sigma_{h-1})}(\sigma_h)}{L_{v_{h}^{(h-1)}(\sigma_1,\ldots,\sigma_{h-1})}(\sigma_h)}v_{h}^{(h-1)}(\sigma_1,\ldots,\sigma_{h-1}),$$
$j=h+1,\ldots,N$. Note that these are well defined linear combinations of $v_j^{(1)}(\sigma_1)$, $j=2,\ldots,N$, hence solutions of \eqref{eq:hyper}, if $\sigma_1>\widetilde\sigma_0$ and $\sigma_h>\sigma^*(L_{v_{h}^{(h-1)}(\sigma_1,\ldots,\sigma_{h-1})})$, $h=2,\ldots,N-1$. Actually, by definition it is clear that, under these conditions, $v_j^{(h)}(\sigma_1,\ldots,\sigma_h)$ is a solution of
\begin{equation*}
\left\{\begin{array}{l}
F_1(\sigma_1) x_1 + \cdots + F_N(\sigma_1) x_N =0\\
\qquad\qquad\qquad\vdots \\
F_1(\sigma_h) x_1 + \cdots + F_N(\sigma_h) x_N =0\\
\end{array}\right..
\end{equation*}
Moreover, for any $1\leq m\leq N-1$ we consider the vector 
\begin{equation}\label{eq:limit}
v_m(\sigma_1,\ldots,\sigma_{m-1},\infty,\ldots,\infty) = \lim_{\sigma_{m}\rightarrow\infty}\cdots\lim_{\sigma_{N-1}\rightarrow\infty} v_N^{(N-1)}(\sigma_1,\ldots,\sigma_{N-1})
\end{equation}
and for simplicity we write $v_N(\sigma_1,\ldots,\sigma_{N-1})=v_N^{(N-1)}(\sigma_1,\ldots,\sigma_{N-1}).$
Note that there exists a finite set of explicit conditions on $\sigma_1,\ldots,\sigma_{N-1}$ for which these limits exist, i.e. there exist $\widetilde\sigma_j$, $j=1,\ldots,N-1$, which depend only on the Dirichlet series $F_1,\ldots,F_N$, such that $v_m(\sigma_1,\ldots,\sigma_{m-1},\infty,\ldots,\infty)$ exists for every $1\leq m\leq N-1$ if $\sigma_l> \widetilde\sigma_l$ for every $l=1,\ldots,N-1$. These conditions actually correspond to the fact that the vector $v_m(\sigma_1,\ldots,\sigma_{m-1},\infty,\ldots,\infty)$ is a generator of the one-dimensional (by \eqref{eq:det}, reordering the functions if needed) vector space defined by the system
\begin{equation*}
\left\{\begin{array}{l}
F_1(\sigma_1) x_1 + \cdots + F_N(\sigma_1) x_N =0\\
\qquad\qquad\qquad\vdots \\
F_1(\sigma_{m-1}) x_1 + \cdots + F_N(\sigma_{m-1}) x_N =0\\
a_1(1) x_1 + \cdots + a_N(1) x_N =0\\
\qquad\qquad\qquad\vdots \\
a_1(N-m) x_1 + \cdots + a_N(N-m) x_N =0\\
\end{array}\right..
\end{equation*}
Hence, in particular, this implies that the definition of $v_m(\sigma_1,\ldots,\sigma_{m-1},\infty,\ldots,\infty)$ is independent from the order of the limits and that $L_{v_m(\sigma_1,\ldots,\sigma_{m-1},\infty,\ldots,\infty)}(\sigma_l)=0$, $l=1,\ldots,m-1$.

We work by induction on $h\in [1,N-2]$. For $h=1$ we fix
$$\sigma_{1,2}>\sigma_{1,1}>\max\left(\sigma^*(L_{v_1(\infty,\ldots,\infty)}), \widetilde\sigma_0\right),$$
and take
$$\eps_1 = \min_{\sigma_{1,1}\leq \sigma\leq \sigma_{1,2},t\in\R}\: |L_{v_1(\infty,\ldots,\infty)}(\sigma+it)|>0$$
and
$$M_1=\max_{1\leq j\leq N}\:\max_{\sigma_{1,1}\leq \sigma\leq \sigma_{1,2},t\in\R} |F_j(\sigma+it)|<\infty.$$
Note that $M_1>0$ by the choice of $\sigma_{1,1}$ and $\sigma_{1,2}$. By \eqref{eq:limit}, we can choose $\beta_1>\sigma_{1,2}$ such that 
$$\norma{v_{1}(\infty,\ldots,\infty)-v_2(\beta_1,\infty,\ldots,\infty)}<\frac{\eps_1}{2\sqrt{N}M_1}.$$
Then, since $v_2(\beta_1,\infty,\ldots,\infty)$ is a solution of \eqref{eq:hyper} with $\sigma=\beta_1$, we have that $L_{v_2(\beta_1,\infty,\ldots,\infty)}(\beta_1)=0$. Moreover for $\sigma_{1,1}\leq \sigma\leq \sigma_{1,2}$ we have, by the triangle and Cauchy-Schwarz inequalities,
$$|L_{v_2(\beta_1,\infty,\ldots,\infty)}(s)|\geq |L_{v_1(\infty,\ldots,\infty)}(s)| - |L_{v_1(\infty,\ldots,\infty)-v_2(\beta_1,\infty,\ldots,\infty)}(s)|\geq \eps_1 - \frac{\eps_1}{2}=\frac{\eps_1}{2}=\delta_1>0.$$

By induction we suppose that for any fixed $1<h\leq N-2$ there exist 
$$\sigma_{1,1}<\sigma_{1,2}<\beta_1<\cdots< \sigma_{h,1}<\sigma_{h,2}<\beta_h$$
and $\delta_h>0$ such that
$$\min_{1\leq l\leq h}\:\min_{\sigma_{l,1}<\sigma<\sigma_{l,2},t\in\R}\: |L_{v_{h+1}(\beta_1,\ldots,\beta_{h},\infty,\ldots,\infty)}(\sigma+it)|>\delta_h.$$
These hypotheses mean that the Dirichlet series $L_{v_{h+1}(\beta_1,\ldots,\beta_{h},\infty,\ldots,\infty)}(s)$, which vanishes for $s=\beta_1,\ldots,\beta_h$, has at least $h$ distinct vertical strips without zeros in the region $1<\sigma<\sigma^*(L_{v_{h+1}(\beta_1,\ldots,\beta_{h},\infty,\ldots,\infty)})$.\\
For the inductive step $h\mapsto h+1$, we take
$$\sigma_{h+1,2}>\sigma_{h+1,1}>\max\Big(\sigma^*(L_{v_{h+1}(\beta_1,\ldots,\beta_{h},\infty,\ldots,\infty)}),\max_{h+1\leq j\leq N}\:\sigma^*(L_{v_{j}^{(h)}(\beta_1,\ldots,\beta_{h})}), \widetilde\sigma_h\Big),$$
$$\eps_{h+1} = \min\left(\delta_h,\min_{\sigma_{h+1,1}\leq \sigma\leq \sigma_{h+1,2},t\in\R} \:|L_{v_{h+1}(\beta_1,\ldots,\beta_{h},\infty,\ldots,\infty)}(\sigma+it)|\right)>0$$
and
$$M_{h+1}=\max_{1\leq j\leq N}\: \max_{\sigma_{1,1}\leq \sigma\leq \sigma_{h+1,2},t\in\R} |F_j(\sigma+it)|<\infty.$$
Note that since $\sigma_{h+1,1}>\sigma_{1,2}$ we have $M_{h+1}>0$. Then we choose $\beta_{h+1}>\sigma_{h+1,2}$ such that 
$$\norma{v_{h+1}(\beta_1,\ldots,\beta_h,\infty,\ldots,\infty)-v_{h+2}(\beta_1,\ldots,\beta_h,\beta_{h+1},\infty,\ldots,\infty)}<\frac{\eps_{h+1}}{2\sqrt{N}M_{h+1}},$$
which exists by definition. Moreover, by the triangle and Cauchy-Schwarz inequalities, we have that
\begin{spliteq*}
|L_{v_{h+2}(\beta_1,\ldots,\beta_{h+1},\infty,\ldots,\infty)}(s)|&\geq |L_{v_{h+1}(\beta_1,\ldots,\beta_{h},\infty,\ldots,\infty)}(s)|\\
&\qquad\qquad-|L_{v_{h+2}(\beta_1,\ldots,\beta_{h+1},\infty,\ldots,\infty)-v_{h+2}(\beta_1,\ldots,\beta_{h+1},\infty,\ldots,\infty)}(s)|\\
&\geq \delta_h-\frac{\eps_{h+1}}{2}\geq \frac{\eps_{h+1}}{2}=\delta_{h+1}
\end{spliteq*}
for any $\sigma_{l,1}\leq \sigma\leq\sigma_{l,2}$, $l=1,\ldots,h+1$.

When $h+1=N-2+1=N-1$ we have just one vector $\bb{c}=v_N(\beta_1,\ldots,\beta_{N-1})\in\C^N\setminus\{\bb0\}$ and the corresponding Dirichlet series $L_{\bb{c}}(s)$ has, as noted above, at least $N-1$ distinct vertical strips without zeros in the region $1<\sigma<\sigma^*(L_{\bb{c}})$. Moreover, note that for every $0\neq z\in\C$ the Dirichlet series $L_{z\bb{c}}(s)$ has the same zeros of $L_{\bb{c}}(s)$ and thus at least $N-1$ distinct vertical strips without zeros in the same region.

\subsection{Proof of Theorem \ref{theorem:func_eq_hole}}
For any $j=1,\ldots,N$, let $\alpha_j$ be a square root of $\omega_j$. Without loss of generality we may suppose that $h=1$ and $k=2$. Note that, since $|\omega_j|=1$ and $\omega_1\neq \omega_2$ then $\alpha_1\neq \pm\alpha_2$ and we may suppose $\alpha_1\notin\R$. It follows that the system of equations
\begin{equation}\label{eq:system1}
\left\{\begin{array}{l}
\re(\alpha_1) x_1 + \cdots + \re(\alpha_N) x_N =0\\
\im(\alpha_1) x_1 + \cdots + \im(\alpha_N) x_N =0\\
\end{array}\right.
\end{equation}
defines a real vector space $V_\infty$ of dimension $N-2\geq 1$ which may be written as
$$V_\infty =\left\{\left(\sum_{j=3}^\infty\left(\frac{\im(\alpha_2)\im(\alpha_1\conj{\alpha_j})}{\im(\alpha_1)\im(\alpha_1\conj{\alpha_2})}-\frac{\im(\alpha_j)}{\im(\alpha_1)}\right) t_j, -\sum_{j=3}^\infty\frac{\im(\alpha_1\conj{\alpha_j})}{\im(\alpha_1\conj{\alpha_2})}t_j ,t_3,\ldots,t_{N}\right)\mid t_3,\ldots, t_{N}\in\R\right\}.$$
Let $v_\infty\in V_\infty$ be the vector corresponding to a fixed choice $(\tau_1,\ldots,\tau_N)\in\R^{N-2}\setminus\{\bb0\}$ and $\bb{c}_0=(\conj{\alpha_1}v_{\infty,1},\ldots,\conj{\alpha_N}v_{\infty,N})$. We take $\sigma_2>\sigma_1>\max(\sigma^*(L_{\bb{c}_0}))$, then, by Theorem 8 of Jessen and Tornehave \cite{jessentornehave}, there exists $\eps>0$ such that $|L_{\bb{c}_0}(s)|>\eps$ for $\sigma_1\leq \sigma\leq \sigma_2$. Moreover, there exists $M>0$ such that $|F_j(s)|\leq M$ for $\sigma_1\leq \sigma\leq \sigma_2$.\\
On the other hand, for any fixed $\sigma>\sigma_2$, the system of equations
\begin{equation}\label{eq:system2}
\left\{\begin{array}{l}
\re(\alpha_1F_1(\sigma)) x_1 + \cdots + \re(\alpha_NF_N(\sigma)) x_N =0\\
\im(\alpha_1F_1(\sigma)) x_1 + \cdots + \im(\alpha_NF_N(\sigma)) x_N =0\\
\end{array}\right.
\end{equation}
defines a real vector space $V_\sigma$ of dimension at least $N-2$. However, since $F_j(\sigma)\rightarrow a_j(1)=1$ as $\sigma\rightarrow\infty$, $j=1,2$, there exists $\sigma_0>\sigma_2$ such that $V_\sigma$ has dimension $N-2$ for every $\sigma>\sigma_0$ and
\begin{spliteq*}
V_\sigma =\Bigg\{\Big(\sum_{j=3}^\infty & \left(\frac{\im(\alpha_2F_2(\sigma))\im(\alpha_1\conj{\alpha_j}F_1(\sigma)\conj{F_j(\sigma)})}{\im(\alpha_1F_1(\sigma))\im(\alpha_1\conj{\alpha_2}F_1(\sigma)\conj{F_2(\sigma)})}-\frac{\im(\alpha_j)F_j(\sigma)}{\im(\alpha_1)F_1(\sigma)}\right) t_j,\\
&\qquad -\sum_{j=3}^\infty\frac{\im(\alpha_1\conj{\alpha_j}F_1(\sigma)\conj{F_j(\sigma)})}{\im(\alpha_1\conj{\alpha_2}F_1(\sigma)\conj{F_2(\sigma)})}t_j ,t_3,\ldots,t_{N}\Big)\mid t_3,\ldots, t_{N}\in\R\Bigg\}.
\end{spliteq*}
Let $v_\sigma\in V_\sigma$ be the vector corresponding to $(\tau_1,\ldots,\tau_N)$, then $\norma{v_\infty-v_\sigma}\rightarrow 0$ as $\sigma\rightarrow\infty$. Therefore there exists $\beta>\sigma_0$ such that, taking $\bb{c}=(\conj{\alpha_1}v_{\beta,1},\ldots,\conj{\alpha_N}v_{\beta,N})$, we have $\norma{\bb{c}_0-\bb{c}}<\frac{\eps}{2\sqrt{N}M}$. Then by \eqref{eq:system2} we have that $L_{\bb{c}}(\beta)=0$ and, by the triangle and Cauchy-Schwartz inequalities, that
$$|L_{\bb{c}}(s)|\geq|L_{\bb{c}_0}(s)|-|L_{\bb{c}-\bb{c}_0}(s)|>\eps-\frac{\eps}{2}=\frac{\eps}{2}$$
for $1\leq \sigma^*(L_{\bb{c}_0})<\sigma_1\leq\sigma\leq\sigma_2<\sigma_0<\beta\leq \sigma^*(L_{\bb{c}})$. Moreover 
$$\Phi(s) = \sum_{j=1}^N \conj{\alpha_j} v_{\beta,j} \Phi_j(s) = \sum_{j=1}^N \conj{\alpha_j} v_{\beta,j} \omega_j\conj{\Phi_j(1-\conj{s})} = \sum_{j=1}^N \alpha_j v_{\beta,j} \conj{\Phi_j(1-\conj{s})} = \conj{\Phi(1-\conj{s})}.$$

\bibliographystyle{amsplain}
\bibliography{biblio}

\end{document}